\documentclass[]{article}
\usepackage[utf8]{inputenc} 
\usepackage[T1]{fontenc}
\usepackage{url}
\usepackage{ifthen}
\usepackage{cite}
\usepackage[cmex10]{amsmath} % Use the [cmex10] option to ensure complicance
\usepackage{cases}
\usepackage{amssymb,mathrsfs,dsfont}
\usepackage{enumerate}
\usepackage{tikz}
\usetikzlibrary{arrows}
\usepackage{mathtools}
\usepackage{scrextend}
\usepackage{graphicx}
\usepackage{caption}
\usepackage{subcaption}
\usepackage{tikz,pgfplots}
\usepackage{enumitem}
\usepackage{setspace}

\usepackage[utf8]{inputenc}
\usepackage[english]{babel}
\usepackage{amsthm}

\newtheorem{theorem}{Theorem}
\newtheorem{lemma}{Lemma}
\newtheorem{proposition}{Proposition}
\newtheorem{corollary}{Corollary}
\theoremstyle{definition}
\newtheorem{remark}{Remark}
\newtheorem{example}{Example}

\DeclareMathOperator{\tr}{tr}
\DeclareMathOperator{\Var}{\mathsf{Var}}
\DeclareMathOperator{\E}{\mathbb{E}}
\DeclareMathOperator{\Cov}{\mathsf{Cov}}

\newcommand{\vect}[1]{\mathbf{#1}} 
\DeclareMathOperator{\Rho}{\boldsymbol{\rho}}
\DeclareMathOperator{\alphaBold}{\boldsymbol{\alpha}}
\DeclareMathOperator{\thetaBold}{\boldsymbol{\theta}}
\DeclareMathOperator*{\argmax}{argmax}

\def\Sx{\Sigma_{\vect{X}}}
\def\Sxsq{\Sx^{\frac{1}{2}}}
\def\Sxsqm{\Sx^{-\frac{1}{2}}}
\def\fisher{\mathrm{I}}
\def\th{\mathrm{th}}
\def\thq{\mathrm{th}\text{-}\mathrm{q}}

\allowdisplaybreaks
\title{Distributed Estimation of Gaussian Correlations}
\author{Uri Hadar and Ofer Shayevitz\thanks{The authors are with the Department of EE--Systems, Tel Aviv University, Tel Aviv, Israel \{emails: urihadar@mail.tau.ac.il, ofersha@eng.tau.ac.il\}. This work was supported by an ERC grant no. 639573. The first author would like to acknowledge the generous support of The Yitzhak and Chaya Weinstein Research Institute for Signal Processing.}}

\begin{document}
\maketitle

\begin{abstract}
 	We study a distributed estimation problem in which two remotely located parties, Alice and Bob, observe an unlimited number of i.i.d. samples corresponding to two different parts of a random vector. Alice can send $k$ bits on average to Bob, who in turn wants to estimate the cross-correlation matrix between the two parts of the vector. In the case where the parties observe jointly Gaussian scalar random variables with an unknown correlation $\rho$, we obtain two constructive and simple unbiased estimators attaining a variance of $(1-\rho^2)/(2k\ln 2)$, which coincides with a known but non-constructive random coding result of Zhang and Berger. We extend our approach to the vector Gaussian case, which has not been treated before, and construct an estimator that is uniformly better than the scalar estimator applied separately to each of the correlations. We then show that the Gaussian performance can essentially be attained even when the distribution is completely unknown. This in particular implies that in the general problem of distributed correlation estimation, the variance can decay at least as $O(1/k)$ with the number of transmitted bits. This behavior, however, is not tight: we give an example of a rich family of distributions for which local samples reveal essentially nothing about the correlations, and where a slightly modified estimator attains a variance of $2^{-\Omega(k)}$.   
\end{abstract}

%\begin{spacing}{0.89}
\section{Introduction and Main Results}
Estimating the parameters of an unknown distribution from its samples is a basic task in many scientific problems. The vast majority of research in this field has been dedicated to the centralized setup, where a number of independent samples are being observed by the estimating entity~\cite{lehmann2006theory}. However, in many cases the data for the estimation task might be collected by remote terminals, who then need to communicate information regarding their observations in order to perform (or improve) estimation. When the budget for communication is limited, the parties must judiciously encode their observations and send a compressed version that is as useful as possible, creating a tension between communication and estimation.  

In this paper, we study the following distributed estimation setup. Let $\vect{X}$ and $\vect{Y}$ be a pair of jointly distributed random vectors taking values in Euclidean spaces of dimensions $d_X$ and $d_Y$ respectively. Assume the distribution of the pair is only known to belong to a given family of distributions, but is otherwise arbitrary. Two remotely located parties, Alice and Bob, draw i.i.d. samples $\{(\vect{X}_i,\vect{Y}_i)\}$ from this distribution, where the $\vect{X}$ component is observed only by Alice and the $\vect{Y}$ component is observed only by Bob. The parties are interested in estimating the set of correlations between the entries of $\vect{X}$ and $\vect{Y}$ using their local samples and limited communication. Specifically, we focus on the regime where the number of samples locally available to each party is essentially {\em unlimited}, but only a {\em fixed} number of $k$ bits can be transmitted on average from (say) Alice to Bob. In this extremal regime there is no coupling between data collection and communication (typically captured by the notion of {\em rate}, of communication bits per data sample), and the only constraint in the system stems from its distributive nature. Moreover, we restrict attention to cases where the correlations cannot be estimated locally (e.g. Gaussian marginals do not depend on the cross-correlation parameters), which further distills the distributive aspect of the problem. 

In what follows we focus mainly on the Gaussian case, i.e., where $\vect{X}$ and $\vect{Y}$ are jointly Gaussian random vectors. We begin our discussion with the scalar $d_X=d_Y=1$ case, where our goal is to estimate the correlation coefficient $\rho$. The only work we are aware of that deals with distributed estimation of the bivariate normal correlation under communication constraints is by Zhang and Berger~\cite{zhang1988estimation}, who studied the problem as an application of a more general result. Using random coding techniques, they proved the existence of an asymptotically unbiased estimator whose variance they provided as a function of the number of samples and the rate $R$ of communication bits per sample. Specializing to our setup by plugging in $k/R$ as the number of samples, the Zhang-Berger variance is given by  
\begin{align}
\Var \hat{\rho}_{ZB} = \frac{R}{k}\left(1 + \rho^2 + \frac{1-\rho^2}{2^{2R}-1} + o(1)\right) .
\end{align}
Since we do not impose a rate constraint in our setup, we can minimize the variance over $R$ to obtain
\begin{align}\label{eq:bergerZeroRate}
\inf_{R>0} \Var \hat{\rho}_{ZB} =  \frac{1}{k} \left( \frac{1-\rho^2}{ 2 \ln 2} + o(1) \right), 
\end{align}
which is attained (not surprisingly) in the zero-rate limit as $R\to 0$. It should be noted that this estimator was not claimed to be optimal in any sense. Furthermore, as the authors themselves indicate, the results in~\cite{zhang1988estimation} apply only to the single scalar parameter case, and it is not clear how to extend this approach to the vector case.  

In this Gaussian scalar setup, addressed in Section~\ref{s:singleCorr}, we introduce the following constructive scheme: Alice sends to Bob the index $J$ of the largest sample among her first $2^k$ samples, and Bob computes the unbiased estimator 
\begin{align}
\hat{\rho}_{\textnormal{max}} = \frac{Y_J}{\E  X_J} \approx \frac{Y_J}{\sqrt{2k\ln{2}}}. 
\end{align}
In Theorem~\ref{thrm:maxEst}, we show that this simple estimator attains the same variance as the non-constructive Zhang-Berger estimator~\eqref{eq:bergerZeroRate}, i.e., 
\begin{align}
\Var \hat{\rho}_{\textnormal{max}} =  \frac{1}{k} \left( \frac{1 - \rho^2}{2 \ln 2} + o(1) \right).
\end{align} 
Then, in preparations for the vector case, we describe a simple variation of this estimator: Alice scans her samples sequentially and finds the index $J$ of the first sample to pass a suitably chosen threshold. She then compresses this index using an optimal lossless variable-length code and sends the encoded version to Bob, who computes an estimator using his corresponding $Y$ sample, in a way similar to the maximum estimator above. This threshold estimator is unbiased, and also attains the Zhang-Berger variance. We note that the maximal/threshold-passing value of a scalar i.i.d. Gaussian sequence has been employed before in problems of writing on dirty paper~\cite{liu2006opportunistic}, \cite{borade2006writing}, and Gaussian lossy source coding~\cite{no2016rateless}. 

We proceed to consider the vector Gaussian setup (Section~\ref{s:multiCorr}). Without loss of generality, we assume that both parties know the distribution of Alice's vector, since she can estimate it arbitrarily well from her local samples and send a sufficiently accurate quantization to Bob with what can be shown to be a negligible cost in communication. In the case where $d_X=1$ and $d_Y>1$ we can trivially extend the scalar estimator by having Alice perform the same encoding (maximal or threshold) and have Bob apply the same type of estimation to each of the entries of $\vect{Y}$ using the single index obtained from Alice. The case of $d_X>1$ and $d_Y=1$ is more interesting. Of course, one could simply estimate each one of the correlations $\rho_\ell$ between $(\vect{X})_\ell$ and $Y$ separately by repeating the scalar method. A worthy goal is therefore to find an estimator that {\em dominates} the scalar approach, uniformly for all correlation values. In Proposition~\ref{pr:sclarNoGood}, we show that performing general linear operations (e.g., whitening the signal) before applying the scalar estimator, does not dominate the scalar approach. We then describe a multidimensional estimator that {\em does} dominate the scalar approach, by generalizing the scalar threshold to an appropriately chosen $d_X$-dimensional {\em stopping set}. We show that the resulting (constructive) estimator $\hat{\Rho}$  attains a total mean squared error that is a function of {\em  the highest correlation only}, and is given by 
\begin{align}
\E \|\hat{\Rho} - \Rho\|^2  \le  \frac{1}{k} \left(\frac{d_X^2 }{2 \ln 2} \min_{\ell \in [d]} \{1 - \rho^2_{\ell}\} + o(1) \right).
\end{align}
This is proved Theorem~\ref{thm:xvec}. We note that the case of $d_X,d_Y>1$ is again a trivial extension of the $d_X>1, d_Y=1$ case. 

Returning to the general non-Gaussian setup (Section~\ref{s:Apps}), we provide two additional results. In Section~\ref{ss:unknwnDists} we show how our estimators above can essentially be used to obtain the {\em same variance guarantees} when $(\vect{X},\vect{Y})$ are {\em arbitrarily distributed}, subject only to uniform integrability fourth moment conditions. This in particular means that one can always get a $O(1/k)$ variance in  distributed correlation estimation with $k$ transmitted bits on average. Recall that in centralized estimation problems, when the family of distributions is sufficiently smooth in the parameter of interest, the Cram\'er–Rao lower bound implies that the optimal estimation variance is $\Theta(1/n)$, where $n$ is the number of samples. Thus, the centralized number of samples required to achieve the same variance as in the distributed case is at least linear in the number of communication bits, i.e., each communication bit is worth at least a constant number of samples. It is perhaps tempting to guess that this relation is fundamental, i.e., that a bit is equivalent to a constant number of samples, hence that the variance cannot decrease faster than $\Omega(1/k)$, assuming that the family of distributions is such that  Bob cannot  estimate the correlations from his local samples. While we conjecture this is true in the Gaussian case, it does not hold in general: In Subsection~\ref{ss:genAdditive} we give an example of a rich family of distributions for which local samples reveal essentially nothing about the correlations, and where the variance of our (slightly modified) estimator is $2^{-\Omega(k)}$.

\subsection{Related Work}
The problem of distributed estimation under communication constraints has been studied in the last couple of decades by several authors. Zhang and Berger~\cite{zhang1988estimation} used random coding techniques to establish the existence of an asymptotically unbiased estimator whose variance is \emph{upper} bounded by a single-letter expression. Their results are limited to a certain family of joint distributions (that must satisfy an \emph{additivity} condition) that depend on a one-dimensional parameter. Ahlswede and Burnashev~\cite{ahlswede1990minimax} gave a multi-letter lower bound on the minimax estimation variance in the one-dimensional case. Han and Amari~\cite{han1995parameter} (see also the survey paper~\cite{amari1998statistical}) suggested a rate constrained encoding scheme, and obtained the likelihood equation based on the decoded statistic. They also showed that the estimation variance asymptotically achieves the inverse of the Fisher information of that statistic. Their results only apply to finite alphabets. Amari~\cite{amari2011optimal} discussed optimal compression in the specific setting of estimating the correlation between two binary sources. He showed that under linear-threshold encoding, there does not exist a single scheme that is uniformly optimal for all correlation values. A similar setup was discussed by Haim and Kochman~\cite{haim2016binary} in the context of hypothesis testing between two correlation values. Zhang {\em et al}~\cite{zhang2013information} provided minimax lower bounds for a distributed estimation setting in which all terminals observe samples from the same distribution. El Gamal and Lai~\cite{el2015slepian} showed that Slepian-Wolf rates are not necessary for distributed estimation over finite alphabets. 

There is a rich literature addressing other aspects of the distributed estimation problem. Xiao {\em et al}~\cite{xiao2004joint} and Lou~\cite{luo2005universal} considered distributed estimation of a location parameter under energy and bandwidth constraints. Gubner~\cite{gubner1993distributed}  considered a Bayesian distributed estimation setting and suggested a local quantization algorithm. Xu and Raginsky~\cite{xu2017information} provided lower bounds on the risk in a distributed Bayesian estimation setting with noisy channels between the data collection terminals and the estimation entity. Braverman {\em et al}~\cite{braverman2016communication} provided lower bounds for some high dimensional distributed estimation problems, again when the samples of all terminals are from the same distribution, e.g. for distributed estimation of the multivariate Guassian mean when it is known to be sparse. The authors of~\cite{schizas2008consensus},~\cite{xiao2006distributed},~\cite{ribeiro2006bandwidth} and~\cite{venkitasubramaniam2007quantization} addressed various distributed estimation setups where the measurements across the sensors are assumed to be independent. 

\subsection{Notations and preliminaries} The standard normal density is denoted by $\phi(x) = e^{-x^2/2}/\sqrt{2 \pi}$, and the tail probability by $Q(x) \triangleq \int_{x}^{\infty} \phi(t)dt$. For $Z \sim \mathcal{N}(0,1)$ the inverse Mills ratio is denoted by 
\begin{align}\label{eq:invMillsDef}
s(x) \triangleq  \E(Z \mid Z>x) = \frac{\phi(x)}{Q(x)}.
\end{align} 
We write $\log$ and $\ln$ for the base $2$ and natural logarithm, respectively. The {\em entropy} of the geometric distribution with parameter $p$ is given by $h_g(p) \triangleq h(p)/p$, 
where $h(p) = -p \log p -  (1 -p) \log (1-p)$ is the binary entropy function. Note that $h_g(p) = -\log(p)(1+ o(1))$ as $p\to 0$. Recall also that any discrete random variable (e.g. in our case, a geometric r.v.) can be losslessly encoded using a prefix-free code with expected length exceeding its entropy by at most one bit~\cite{cover2006elements}. Since in the setups we consider the entropy grows large, this excess one bit has vanishing effect on our results, hence for the sake of readability we disregard it throughout. 

For any natural $n$ we denote $[n] \triangleq \{1,\ldots, n\}$. For a vector $\vect{v}$, the $i$-th coordinate is denoted by $(\vect{v})_i$. Similarly, $(M)_{ij}$ denotes the $ij$-th entry of the matrix $M$. The $d \times d$ identity matrix is denoted by $\vect{I}_d$. We use the standard order notation; in the following, $f$ and $g$ are positive functions with discrete or continuous domain. We write $f = o(g)$ to indicate that $\lim f/g = 0$, and $f = O(g)$ to indicate that $\limsup f/g < \infty$, where the arguments and implied limits should be clear from the context.  Writing $f = \Omega(g)$ means that $g = O(f)$, and $f = \Theta(g)$ means that both $f = O(g)$ and $f = \Omega(g)$. 

Given a statistic $T$, and a scalar parameter $\theta$ we wish to estimate, The Fisher information of estimating $\theta$ from $T$ (see e.g. \cite{lehmann2006theory}) is given by 
\begin{align}
	\fisher_T(\theta) \triangleq \E\left[\left(\frac{\partial \log f(T\mid \theta)}{\partial\theta}\right)^2\right],
\end{align}
where $f(t\mid \theta)$ is the p.d.f. of $T$ for the given value of $\theta$.  
The Cram\'er–Rao lower bound (CRLB) states that,  under some regularity conditions (see e.g. \cite{lehmann2006theory}) that are trivially satisfied in our setups, any unbiased estimator $\hat{\theta}=\hat{\theta}(T)$ of $\theta$ satisfies 
\begin{align}\label{eq:crlbsclr}
\Var \hat{\theta} \ge 1/\fisher_T(\theta).
\end{align}
An estimator $\hat{\theta}$ that satisfies \eqref{eq:crlbsclr} with equality is said to be \emph{efficient}. We emphasize that the efficiency is with respect to the statistic $T$ by saying it is \emph{efficient given $T$}. The estimators and statistics in this paper depend on the number of communicated bits, $k$. We call an estimator $\hat{\theta}$  {\em asymptotically efficient given $T$} if  $\E \hat{\theta} \to  \theta$, and $ \fisher_T(\theta)\cdot \Var\hat{\theta}  \to 1$ as $k\to \infty$. The estimated parameter may be vector valued, in which case $\fisher_T(\thetaBold)$ is a matrix given by
\begin{align}
\fisher_T(\thetaBold) \triangleq \E\left[\left(\frac{\partial \log f(T\mid \thetaBold)}{\partial\thetaBold}\right)^T\cdot \left(\frac{\log f(T\mid \thetaBold)}{\partial\thetaBold}\right)\right], 
\end{align}
and the CRLB reads $\Cov \hat{\thetaBold} \ge \fisher_T^{-1}$ where the inequality is in the positive semidefinite sense. In the vector case we say that an estimator $\hat{\thetaBold}(T)$  is asymptotically efficient if the estimator  $\vect{v}^T\hat{\thetaBold}(T)$ of $\vect{v}^T \thetaBold$ is asymptotically efficient for any $\vect{v} \in \mathbb{R}^{\dim(\thetaBold)}$. We note that since the aforementioned regularity conditions are satisfied in our Gaussian setups, then (asymptotic) efficiency of an estimator implies that it is  \emph{(asymptotically) minimum variance unbiased}.

The Fisher information matrix of a Gaussian vector with mean $\mu$ and covariance matrix $\Sigma$, where both are functions of a parameter vector $\thetaBold$, is given by  (see e.g. \cite{kay1993fundamentals})
\begin{align}\label{eq:generalFisher}
(\fisher)_{ij} \!=\! \frac{\partial \mu^T}{\partial (\thetaBold)_i} \Sigma^{-1}\! \frac{\partial \mu}{\partial (\thetaBold)_j} \!+\! 
\frac{1}{2} \tr \left(  \Sigma^{-1}\! \frac{\partial \Sigma}{\partial (\thetaBold)_i} \Sigma^{-1}\! \frac{\partial \Sigma}{\partial (\thetaBold)_j} \right)\!.
\end{align}
A common setup throughout is where a parameter is estimated from $(\vect{X},\vect{Y})$  where $\vect{Y}|\vect{X}$ is Gaussian, and the distribution of $\vect{X}$ does not depend on the parameter. In this case we have
\begin{align}\label{eq:generalFisher2}
\fisher_{\vect{X},\vect{Y}} = \E_{\vect{X}} \fisher_{\vect{Y}|\vect{X}}
\end{align}
where $\fisher_{\vect{Y}|\vect{X}}$ is obtained via \eqref{eq:generalFisher}.

\section{Estimating a single correlation}\label{s:singleCorr}
In this section, we consider the case where $X$ and $Y$ are both scalar, jointly Gaussian r.v.s, with unknown parameters satisfying only $\E{X^2}, \E{Y^2}  < u$ for some known $u$.  Since the number of local samples available to Alice and Bob is unlimited, they can both estimate their own mean and variance arbitrarily well (taking $u$ into account) and normalize their samples accordingly. Hence, without loss of generality we can assume that $X,Y\sim \mathcal{N}(0,1)$, and that the only unknown parameter 
is their correlation coefficient $\rho$. This model can be written as
\begin{align}\label{eq:scalarModel}
Y = \rho X + \sqrt{1-\rho^2}Z
\end{align}
where $Z \sim \mathcal{N}(0,1)$ is statistically independent of $X$. 

Alice, who observes the i.i.d. samples $\{X_i\}$, can transmit $k$ bits on average to Bob, who observes the corresponding $\{Y_i\}$ samples and would like to obtain a good estimate of $\rho$ in the mean squared error sense. We note that the conditional Fisher information of $\rho$ from $Y$, given that $X=x$, is 
\begin{align}\label{eq:fisher_scalar_basic}
\fisher_{Y | X=x}(\rho) = \frac{(1 - \rho^2) x^2 + 2 \rho^2}{(1 - \rho^2)^2}, 
\end{align}
which is linear in $x^2$. This motivates using an estimator based on a measurement for which $|x|$ is as large as possible. The same can also be intuitively deduced from~\eqref{eq:scalarModel}, since if one controls $X$, then picking it as large as possible would ``maximize the SNR''. For simplicity, we look at large positive values of $x$ rather than large value of $|x|$. Our derivations can be easily modified to hold in the latter case (with one extra bit describing the sign) without affecting the results.

\subsection{Max estimator}\label{sss:max} 
Following the heuristic discussion above, consider the following scheme. Given the constraint $k$ on the expected number of communication bits, Alice looks at her first $2^k$ samples, finds the maximal one, and sends its index
\begin{align}
J = \argmax_{i \in [2^k]} X_i
\end{align}
to Bob, using exactly $k$ bits. Bob now looks at $Y_J$, his sample that corresponds to the same index, which we refer to as the {\em co-max}\footnote{This is also known in the literature as the {\em max concomitant}, see e.g. \cite{david2004order}}. If Bob were in possession of $X_J$ as well, and observing the model~\eqref{eq:scalarModel} again, a natural estimator for $\rho$ he could have used is $Y_J/X_J$. In fact, it can be shown that this estimator is an approximated solution to the maximum likelihood equation, which is third a degree polynomial in this case (see appendix~\ref{apn:MLE}). However, since $X_J$ is not available, Bob uses the estimator 
\begin{align}
\hat{\rho}_{\textnormal{max}} = \frac{Y_J}{\E  X_J}
\end{align}
that depends only on $J$ (communicated by Alice) and on his own samples. The following Theorem shows that this simple estimator attains the same variance as the non-constructive Zhang-Berger estimator~\eqref{eq:bergerZeroRate}, and also that knowing the value of $X_J$ does not help. 
\begin{theorem}\label{thrm:maxEst}
	The estimator $\hat{\rho}_{\textnormal{max}}$ is unbiased with 
	\begin{align} \label{eq:rhoest3inbits}
	\Var \hat{\rho}_{\textnormal{max}} =  \frac{1}{k} \left( \frac{1 - \rho^2}{2 \ln 2} + o(1) \right)	
	\end{align} 
	where $k$ is the number of transmitted bits. Furthermore, $\hat{\rho}_{\textnormal{max}}$ is asymptotically efficient given $(X_J,Y_J)$.
\end{theorem}
\begin{proof}
It is easy to check that $\hat{\rho}_{\textnormal{max}}$ is unbiased. In order to compute its variance, we need to compute the mean and variance of $X_J$, which is the maximum of $2^k$ i.i.d. standard normal r.v.s. From extreme value theory (see e.g. \cite{david2004order}) applied to the normal distribution case, we obtain: 
	\begin{align}
	&\E X_J = \sqrt{2 \ln (2^k)}(1 + o(1))\\
	&\E X_J^2 = 2 \ln{(2^k)} (1 + o(1))\\
	&\Var X_J = O \left( \frac{1}{\ln (2^k)} \right).
	\end{align}
Therefore, we have that 
	 	\begin{align}
\Var \hat{\rho}_{\textnormal{max}} 
&= \frac{1}{(\E X_J)^2} \Var (\rho X_J + \sqrt{1-\rho^2}Z) \\
&= \frac{1}{(\E X_J)^2} (\rho^2 \Var X_J + 1-\rho^2) \\
&= \frac{1}{2k \ln 2} (1-\rho^2 + o(1)).
\end{align}

Now, recalling~\eqref{eq:fisher_scalar_basic}, the Fisher Information of $\rho$ from $(X_J,Y_J)$ is given by 
\begin{align}\label{eq:maxFish}
	\fisher_{X_J Y_J} (\rho) &= \frac{(1 - \rho^2) \E  X_J^2 + 2 \rho^2}{(1 - \rho^2)^2}\\
	&=2k\ln{2} \left( \frac{1 }{1 - \rho^2} + o(1) \right),
\end{align}
and hence $\hat{\rho}_{\textnormal{max}}$ is asymptotically efficient given $(X_J,Y_J)$. 
\end{proof}
Theorem \ref{thrm:maxEst} suggests that using a better estimator of $X_J$ in lieu of its expectation (by having Alice send some quantization of $\hat{X}_J$ and having Bob compute $\hat{\rho} = Y_J / \hat{X}_J$) would not improve the performance asymptotically, as $\hat{\rho}_{\textnormal{max}}$ is optimal among all unbiased estimators that use both the max and co-max. In Section \ref{ss:genAdditive}, we will see that this observation does not extend to some other additive models.

We note that the random coding Zhang-Berger estimator only deals with the scalar case, and as the authors themselves indicate~\cite{zhang1988estimation}, it remains unclear whether it could be extended to the the case of multiple correlations. In contrast, our constructive approach can also be naturally extended to the multidimensional case. To that end, it is instructive to first describe a simple variation of our scalar estimator. 

\subsection{Threshold estimator}\label{ss:sclrThresh}
We now introduce a simple modification to max estimator that will be useful in the sequel. Instead of taking the maximum of a fixed number of measurements, Alice sequentially scans her samples until she finds a sample that exceeds some fixed threshold, to be determined later. She then sends the index of this sample to Bob, who proceeds similarly to the max method. The main difference is that using the max method Alice sends a fixed number of bits, whereas using the threshold method she sends a random number of bits. In this subsection, we introduce and analyze the threshold estimator and demonstrate that it is asymptotically equivalent to the max estimator, in terms of how the estimation variance is related to the expected number of bits transmitted. As mentioned above, the main motivation for studying the threshold estimator is that in contrast to the max estimator, it can be naturally extended to the multidimensional case.

Precisely, let 
\begin{align}
J = \min\{i:X_i > t\}, 
\end{align}
and consider the estimator
\begin{align}
\hat{\rho}_{\th} = \frac{Y_J}{\E X_J}.
\end{align}
Note that the index $J$ is distributed geometrically with parameter $p = \Pr(X > t) = Q(t)$. Alice can therefore represent $J$ using a prefix-free code (e.g., Huffman) with at most $h_g(p)+1$ bits on average, where $h_g(p)$ is the entropy of this geometric distribution~\cite{cover2006elements}. For brevity of exposition, we assume that the expected number of bits is exactly $k=h_g(p)$, as this does not affect the asymptotic behavior. Therefore, to satisfy the communication constraint the threshold must be set to 
\begin{align}
t = Q^{-1}(h_g^{-1}(k)). 
\end{align}
We later show that $t = \sqrt{2k \ln{2}}(1 + o(1))$ as $k$ grows large.  The following Theorem shows that as the max estimator, the threshold estimator also attains the same variance as the non-constructive Zhang-Berger estimator~\eqref{eq:bergerZeroRate}, and also that knowing the value of $X_J$ again does not help. 
%The Fisher information of $(X_J,Y_J)$, again, is given by
%\begin{align}\label{eq:fi3}
%\fisher_{X_J Y_J} = \frac{(1 - \rho^2) \E X_J^2 + 2 \rho^2 }{(1 - \rho^2)^2}.
%\end{align}
\begin{theorem}\label{thrm:stEst}
	The estimator $\hat{\rho}_{\th}$ is unbiased with
	\begin{align} \label{eq:rhoest2inbits}
	\Var \hat{\rho}_{\th} =  \frac{1}{k} \left( \frac{1 - \rho^2}{ 2 \ln 2} + o(1)  \right)
	\end{align} 
	where $k$ is the expected number of transmitted bits.  Furthermore, $\hat{\rho}_{\th}$ is asymptotically efficient given $(X_J,Y_J)$.
\end{theorem}
\begin{proof}
	It is immediate to verify that $\hat{\rho}_{\th}$ is unbiased. We have from \eqref{eq:invMillsDef} that $\E X_J = s(t)$, and straightforward calculations give $\E X_J^2 = 1 + ts(t)$. Also it is known that $t \le s(t) \le t + t^{-1}$, and that (see e.g. \cite{small2010expansions}) 
	\begin{align}\label{eq:tailMeanExp}
	\frac{1}{s(t)} = \frac{1}{t} - \frac{1}{t^3}  + \frac{3}{t^5} + O \left( \frac{1}{t^7} \right).
	\end{align}
	Combining the above yields 
	\begin{align}\label{eq:thr_moments}
	&\E X_J^2 = t^2 (1+o(1)), \quad \Var X_J = 1/t^2 + O(1/t^4). 
	\end{align}
	
	Let us now express the threshold $t$ in terms of $k$. We have $h_g(p) = - \log (p)(1 + o(1))$ as $p \rightarrow 0$, and also that $-\ln Q(t) = \frac{t^2}{2}(1+o(1))$. Therefore the expected number of bits sent by Alice is
	\begin{align}
	k &= h_g(Q(t))\\
	&= - \log (Q(t)) (1 + o(1)) \\
	&= t^2 \left( \frac{1}{2 \ln 2}+o(1) \right), 
	\end{align}
	which yield  $t = \sqrt{2k \ln{2}}/(1 +  o(1))$. Combining this with~\eqref{eq:thr_moments} and recalling the model~\eqref{eq:scalarModel}, we obtain 
	\begin{align} 
	\Var \hat{\rho}_{\th} 
	&= \frac{1}{s^2} \left(\rho^2 \Var X_J +  1 - \rho^2 \right)\\ \label{eq:singleCorrMSEwrtt}
	&= \frac{1 - \rho^2}{t^2} (1 + o(1))\\\
	& = \frac{1}{k} \left( \frac{1 - \rho^2}{ 2 \ln 2} + o(1)  \right).
	\end{align}
	Recalling~\eqref{eq:fisher_scalar_basic}, the Fisher information is given by 
	\begin{align}
	\fisher_{X_J Y_J} &= \frac{t^2}{1 - \rho^2} (1+o(1)) =\frac{2k\ln{2}}{1-\rho^2}(1+o(1)), 
\end{align}
concluding the proof. 
\end{proof}

Note that unlike the maximum estimator, the threshold estimator's variance admits an exact non-asymptotic expression:
\begin{align}
\Var \hat{\rho}_{\th} = \frac{1}{s^2(t)}(1 - \rho^2 (s(t)-t)\cdot s(t)), 
\end{align}
where $t = Q^{-1}(h_g^{-1}(k))$.

\section{Estimating multiple correlations}\label{s:multiCorr}
We proceed to address the more challenging multidimensional case where $\vect{X},\vect{Y}$ are jointly Gaussian random vectors with unknown parameters. As in the scalar case, we only assume that the variances of all the entries of both $\vect{X}$ and $\vect{Y}$ are bounded by some known constant, hence Alice and Bob can compute the means and variances of their samples, and normalize them accordingly. Thus, without loss of generality we can assume that all the entries of $\vect{X}$ and $\vect{Y}$ have zero mean and unit variance. In fact, for the same reasons we can assume that Alice knows $\Cov{\vect{X}}$ and Bob knows $\Cov{\vect{Y}}$. 

As before, Alice observes the i.i.d. samples $\{\vect{X}_i\}$ and can transmit $k$ bits on average to Bob, who observes the corresponding $\{\vect{Y}_i\}$ samples and would like to obtain a good estimate of $\E \vect{Y} \vect{X}^T$, the collection of all the correlations between the different entries of $\vect{X}$ and $\vect{Y}$. For simplicity, our performance measure will be the expected sum of squared estimation errors across all such correlations. 

Below we discuss the two extremal setups: The case where $X$ is a scalar and $\vect{Y}$ is a vector, and the opposite case where $\vect{X}$ is a vector and $Y$ is a scalar. This is sufficient since estimators for the general setup where both $\vect{X}, \vect{Y}$ are vectors are straightforward to construct by combining the two extremal setups, hence discussing this more general setup adds no useful insight. Clearly, the scalar methods suggested in Section~\ref{s:singleCorr} can be directly applied to the multidimensional case, by allocating the bits between the tasks of estimating each correlation separately. It is therefore interesting to try and find a truly multidimensional scheme that \emph{dominates} the scalar method, i.e., performs at least as good uniformly for all possible values of the correlations. 

\subsection{$X$ is a scalar, $\vect{Y}$ is a vector}\label{subsec:xscalaryvec}
%\subsection{$X$ is a scalar, $\vect{Y}$ is a vector}\label{ss:Yvec}
Suppose $(X,\vect{Y})$ are jointly Gaussian, where $X \sim \mathcal{N}(0,1)$, $\vect{Y}\sim \mathcal{N}(\vect{0},\Sigma_{\vect{Y}})$ is a $d$-dimensional (column) vector, and $\Sigma_{\vect{Y}}$ has an all-ones diagonal and is known to Bob, who is interested in estimating the column correlation vector 
\begin{align}
\Rho = \E \vect{Y} X = [\rho_1,\ldots, \rho_d]^T. 
\end{align}
The natural extension of the two scalar methods of Section~\ref{s:singleCorr} to this case is obvious. Here we analyze the threshold method, yet the max method is as simple and would yield the same results. Alice waits until $X_i$ passes a threshold $t>0$ and transmits the resulting index
\begin{align}
J = \min\{i:X_i > t\} 
\end{align}
to Bob, where $t = Q^{-1}(h_g^{-1}(k))$. The estimator is then
\begin{align}\label{eq:xScYVecEst}
\hat{\Rho} = \frac{1}{\E X_J}\vect{Y}_J = \frac{1}{s(t)}\vect{Y}_J,
\end{align}
which is an unbiased approximation of the maximum likelihood estimator (see Appendix~\ref{apn:MLE}).
\begin{theorem}\label{thm:yvec}
	The estimator $\hat{\Rho}$ in \eqref{eq:xScYVecEst} is unbiased with 
	\begin{align}
	\tr \Cov \hat{\Rho} = \frac{1}{k}\left( \frac{1}{2 \ln 2}\sum_{\ell=1}^d (1-\rho_{\ell}^2)  + o(1)\right)
	\end{align}
	where $k$ is the expected number of transmitted bits. Furthermore, $\hat{\Rho}$  is asymptotically efficient given  $(X_J,\vect{Y}_J)$. 
\end{theorem}
\begin{proof}
	This is simple consequence of Theorem~\ref{thrm:stEst}, except for asymptotic efficiency which we prove in Appendix~\ref{apn:yVecthmProof}.
\end{proof}
This method (trivially) dominates the scalar method applied separately to each of the correlations, as the latter would yield $\sum \Var \hat{\rho}_i =   \frac{1}{k}\left( \frac{d}{2 \ln 2}\sum_{\ell=1}^d (1-\rho_{\ell}^2)  + o(1)\right)$.

\subsection{$\vect{X}$ is a vector, $Y$ is a scalar}\label{subsec:xvecyscalar}
%\subsection{$\vect{X}$ is a vector, $Y$ is a scalar}
Consider the setup where $(\vect{X},Y)$ are jointly Gaussian where $Y \sim \mathcal{N}(0,1)$ and $\vect{X}$ is a $d$-dimensional (column) vector $\sim \mathcal{N}(\vect{0},\Sx)$ where $\Sx$ is known to Alice and has an all-ones diagonal. Alice observes $\{\vect{X}_i\}$ and transmits $k$ bits to Bob on average, who observes $\{Y_i\}$ and wishes to estimate the \emph{row} vector 
\begin{align}
\Rho = \E Y \vect{X}^T = [\rho_1, \ldots, \rho_d].
\end{align}
The model can be written as (see e.g. \cite{rencher2003methods})
\begin{align}\label{eq:twoTwoModel}
Y = \Rho \Sigma_\vect{X}^{-1} \vect{X} + \sigma Z
\end{align}
where $Z \sim \mathcal{N}(0,1)$  is independent of $\vect{X}$, and $\sigma^2 = 1 - \Rho \Sx^{-1} \Rho^T$. 

%\paragraph{Scalar based approach (is not good)}
A naive extension of the scalar method to this setup would be to allocate the bits between the correlations and apply the scalar (max or threshold) scheme $d$ times, using the fact that the model can also be written as
\begin{align}
Y = \rho_{\ell} (\vect{X})_{\ell} + \sqrt{1 - \rho_{\ell}^2} Z
\end{align}
for any $\ell \in [d]$. One could suggest to improve performance by having Alice locally perform some general linear operation on $\vect{X}$ before applying the scalar method, then having Bob perform the inverse operation. While this can indeed help for certain correlation values, it cannot improve the performance uniformly, even if the linear operation can depend on $\Sx$ (hence can e.g.  whiten $\vect{X}$). See Appendix~\ref{apn:SclrLinTrnsfsm} for details.

%\paragraph{Joint estimation approach (is good)}
We now introduce an estimator that {\em does} dominate the scalar method. In fact, the mean squared error attained by this estimator is dictated by the single ``best'' entry of $\Rho$, namely by the highest correlation only. Our method is based on replacing the scalar one-dimensional threshold by $d$-dimensional {\em stopping sets} $A_1, \ldots, A_d \subset \mathbb{R}^d$. Similarly to the scalar case, Alice waits until $\vect{X}_i \in A_1$ for the first time, then again until $\vect{X}_i \in A_2$, and so on\footnote{The communication cost can be slightly improved if Alice first seeks $\vect{X}_i$ that lies in the union of all sets, then $\vect{X}_i$ that lies in the union of the remaining sets, and so on. The difference is negligible for small $\Pr(A_1), \ldots, \Pr(A_d)$.} until $\vect{X}_i \in A_d$.
Alice then describes the resulting indices $J_1, \ldots J_d$ to Bob using an optimal variable-rate prefix-free code of expected length equal to the entropy of the associated geometric distribution (again, we neglect the excess one bit). Defining Alice's corresponding sample matrix $\vect{X}_{\vect{J}} = [\vect{X}_{J_1}, \ldots, \vect{X}_{J_d}]\in\mathbb{R}^{d\times d}$, Alice creates some quantization  $\hat{\vect{X}}_{\vect{J}}$ of $\vect{X}_{\vect{J}}$, as further discussed below. Writing $Y_{\vect{J}} = [Y_{J_1},  \ldots, Y_{J_d}]\in\mathbb{R}^d$ for the corresponding sample vector on Bob's side, we consider the estimator 
\begin{align}\label{eq:twoPrmEst}
\hat{\Rho} = Y_{\vect{J}} \hat{\vect{X}}_{\vect{J}}^{-1} \Sx 
\end{align}
Note that in order to compute this estimator, Bob needs to know Alice's covariance matrix $\Sx$.  Recall however that we have assumed without loss of generality that this is in fact a correlation matrix, hence all its entries have absolute value at most $1$. Using a uniform quantizer of $[-1,1]$ with (say) $\sqrt{k}$ bits, each entry of this matrix can be described to Bob with a resolution of roughly $2^{-\sqrt{k}}$, using only $d^2\sqrt{k}$ bits overall. It is simple to check that this results in a negligible cost both in communication and in the mean squared error, and hence we disregard this issue below. 

The general task is the following. Given a specified average number of bits $k$, find some quantization scheme $\vect{X}_{\vect{J}} \to \hat{\vect{X}}_{\vect{J}}$ using $k_q$ bits per entry, and sets $A_1,  \ldots, A_d\in \mathbb{R}^d$, that 
\begin{align}\label{eq:xVecGenOpt1}
&\textnormal{minimize } \E\|Y_{\vect{J}} \hat{\vect{X}}_{\vect{J}}^{-1} \Sx  - \Rho\|^2 \\  \label{eq:xVecGenOpt2}
&\textnormal{subject to } \sum_{\ell=1}^{d} h_g(\Pr(\vect{X} \in A_{\ell})) + d^2\cdot k_q = k
\end{align}

Since the model~\eqref{eq:twoTwoModel} is linear with $d$ parameters, it is clear that we need at least $d$ different samples in order to obtain an estimator with a vanishing mean squared error. Furthermore, since Alice is given some control over the choice of $\vect{X}$ via her ability to pick samples from a large random set, it makes sense to try and make the problem as ``well-posed'' as possible, e.g., by striving to make the matrix $\vect{X}_{\vect{J}}$ have the smallest possible condition number while satisfying the communication constraints, which essentially dictate the number of samples we can choose from. A reasonable choice is therefore to try and make $\vect{X}_{\vect{J}}$ as diagonal as possible, by waiting each time for one coordinate to be strong and the others weak. 

To make the problem tractable we apply the rationale above to a whitened version of $\vect{X}$, which allows us to directly compute the stopping probability. Let 
\begin{align}
\vect{W} = \Sx^{-\frac{1}{2}} \vect{X}\sim \mathcal{N}(\vect{0},\vect{I}_{d})
\end{align}
be the whitened version of $\vect{X}$, and $\{\vect{W}_i = \Sx^{-\frac{1}{2}} \vect{X}_i\}_i$ the associated whitened samples. We define the stopping sets 
\begin{align}
A^w_{\ell} = \left\{\vect{w} \in \mathbb{R}^d \ :  |w_{\ell}|>a,\, |w_j|<b \ \forall \ j\neq \ell \right\},
\end{align}
and the corresponding time indices
\begin{align}
J^w_\ell = \min\{i > J^w_{\ell-1} : \vect{W}_i \in A^w_\ell\}
\end{align}
for $\ell\in[d]$, with $J^w_0 = 0$ by definition. 

Note that by construction, $\Pr(\vect{W}\in A^w_\ell)  = 2Q(a)(1-2Q(b))^{d-1}$ for any $\ell$. Alice then creates the matrix
\begin{align}
\vect{W}_{\vect{J}} = [\vect{W}_{J^w_1}, \ldots, \vect{W}_{J^w_d}] \in \mathbb{R}^{d\times d},
\end{align}
and transmits to Bob the indices $J^w_1, \ldots, J^w_d$ using 
\begin{align}\label{eq:xVecabkl}
 k_l = h_g \left( 2 Q(a)  (1 - 2Q(b))^{d-1} \right)
 \end{align} 
bits per index on average, and $\hat{\vect{W}}_{\vect{J}}$, which is a quantized version of $\vect{W}_{\vect{J}}$.

Note that in this method, in contrast to the ones considered thus far, Alice transmits to Bob some information regarding the actual values of her observations, rather than their locations alone. The reason is that the variance of the off-diagonal entries of $\vect{W}_{\vect{J}}$ do not vanish as $k$ gets large. Nevertheless, we will show that a very simple quantizer using only a negligible number of bits is enough to represent $\vect{W}_{\vect{J}}$ with sufficient accuracy for our purposes. Precisely, Alice quantizes $\vect{W}_{\vect{J}}$ using exactly $k_q$ bits per entry, as follows: The diagonal entries are truncated to a maximal absolute value of $c = \sqrt{3} a$, and the double segment $[-c,-a]\cup [a, c]$ is uniformly quantized into $2^{k_q}$ levels.   Off-diagonal entries, that all lie in the segment $[-b, b]$, are uniformly quantized into $2^{k_q}$ levels. 

Given a communication constraint of $k$ bits on average, we need to choose $k_l,k_q$ that satisfy
\begin{align}\label{eq:xVecTotBits}
d \cdot k_l + d^2 \cdot k_q = k,
\end{align}
and thresholds $a,b$ that satisfy \eqref{eq:xVecabkl}. Furthermore, for reasons explained in the proof of Theorem~\ref{thm:xvec}, we need both $a^2$ and $(a-b)^2$ to increase with $k_l$, and $k_q, k_l$ to satisfy $k_l = k(1/d - o(1))$ and $k_l 2^{-k_q} \rightarrow 0$. 
One such choice is 
\begin{align}
&k_l = \frac{1}{d} \left( \sqrt{k+1} - 1 \right)^2, \quad k_q = \sqrt{\frac{4 k_l}{d^3}},
\end{align}
and 
\begin{align}\label{eq:approxab}
a &= Q^{-1} \left(\frac{h_g^{-1}(k_l)}{2(1-2Q(b_0))^{d-1}}\right), \quad b = b_0, 
\end{align}
for some small fixed $b_0$.

After receiving $J^w_1, \ldots, J^w_d$ and $\hat{\vect{W}}_{\vect{J}}$, Bob creates the vector
 \begin{align}
 Y_{\vect{J}} = [Y_{J^w_1}, \ldots, Y_{J^w_d}]
 \end{align}
 and performs estimation. The model \eqref{eq:twoTwoModel}  can be written as
\begin{align}
Y = \Rho \Sx^{-\frac{1}{2}} \vect{W} +\sigma Z,
\end{align}
and thus the estimator is 
\begin{align}\label{eq:twotwoest}
\hat{\Rho} = Y_{\vect{J}} \hat{\vect{W}}_{\vect{J}}^{-1} \Sx^{\frac{1}{2}}.
\end{align} 
\begin{theorem}\label{thm:xvec}
	The estimator $\hat{\Rho}$ in~\eqref{eq:twotwoest} satisfies
	\begin{align}\label{eq:xVectrCov}
	\E \|\hat{\Rho} - \Rho\|^2  \le  \frac{1}{k} \left(\frac{d^2 }{2 \ln 2} \min_{\ell \in [d]} \{1 - \rho^2_{\ell}\} + o(1) \right),
	\end{align}
	where $k$ is the expected number of transmitted bits. Furthermore, $\hat{\Rho}$ is asymptotically efficient given $(\vect{W}_{\vect{J}}, Y_{\vect{J}})$.
\end{theorem}
We prove this theorem in the next subsection. 
\begin{corollary}
	$\hat{\Rho}$ in~\eqref{eq:twotwoest} dominates the scalar estimator. 	
\end{corollary}
\begin{proof}
	Allocating $k/d$ bits per correlation and using the scalar estimator (max or threshold),  results in a sum of variances
	\begin{align}\label{eq:naiveSclrVar}
	\frac{1}{k} \left(\frac{d^2 }{2 \ln 2} \frac{1}{d} \sum_{\ell \in [d]} (1 - \rho^2_{\ell}) + o(1) \right)
	\end{align}
	which is greater than \eqref{eq:xVectrCov} for all values of $\rho_1, \ldots, \rho_d$ (except when they are all equal). One could also use a nonuniform bit allocation for the scalar estimation, in which case the average in \eqref{eq:naiveSclrVar} would be replaced by a weighted average, which also is always greater than the minimum.
\end{proof}

\begin{remark}
	Theorem~\ref{thm:xvec} implies in particular that when (say) $|\rho_1|=1$, then the variance of our estimator decays faster than $\Omega(1/k)$. This is intuitively reasonable, since in this case $Y$ is equal to $\pm X_1$, hence $\Sx$ itself provides all the information about $\Rho$, which can be locally computed by Alice and communicated to Bob with variance of $2^{-\Omega(k)}$. Note however that Alice {\em cannot know} that $|\rho_1|=1$, and neither can Bob (though he may have good reason to suspect so), hence it is still a bit surprising that our estimator allows this situation to nevertheless be exploited. 
\end{remark}
\begin{remark}
It is interesting to compare the performance of the estimator discussed in this subsection, to the performance of the estimator in the other extremal setup of Subsection~\ref{subsec:xscalaryvec}, where $X$ is a scalar and $\vect{Y}$ is a vector. While both dominate the naive scheme of applying the scalar method $d$ times, neither dominates the other. The difference between them, essentially, is the difference between $\sum (1 - \rho_{\ell}^2)/d$ and $d \min \{1 - \rho_{\ell}^2\}$. For example, the former outperforms the latter if all correlations are equal, whereas the latter outperforms the former if any of the correlations is $\pm 1$. 
\end{remark}

\subsection{Proof of Theorem~\ref{thm:xvec}}
Consider the estimator
\begin{align}\label{eq:unrlzblEst}
\hat{\Rho}_0 = Y_{\vect{J}} \vect{W}_{\vect{J}}^{-1} \Sx^{\frac{1}{2}}.
\end{align}
Note that this estimator uses the non-quantized $\vect{W}_{\vect{J}}$ which cannot be described to Bob with a finite number of bits, and hence is unrealizable. Nevertheless, as the following lemma shows, the loss incurred by employing $\hat{\Rho}$ instead, which uses the quantized $\vect{W}_{\vect{J}}$, is small. 
\begin{lemma}\label{lm:lemma1}
	For any $a,b$ such that $a > d (b + 1)$, 
	\begin{align}
	\E\|\hat{\Rho} - \Rho\|^2 \le \E\|\hat{\Rho}_0 - \Rho\|^2 + (2d)^6  \left( e^{-\frac{a^2}{2}}  +  2^{-k_q} \right)
	\end{align}
where $k_q$ bits are used to represent each entry in $\hat{\vect{W}}_{\vect{J}}$.
\end{lemma}
\begin{proof}
	See Appendix~\ref{apn:lm1Proof}.
\end{proof}
The estimator \eqref{eq:unrlzblEst} can be written as 
\begin{align}
\hat{\Rho}_0 = \Rho  +\sigma Z_{\vect{J}} \vect{W}_{\vect{J}}^{-1} \Sx^{\frac{1}{2}}.
\end{align}
where $Z_{\vect{J}} = [Z_{J^w_1}, \ldots, Z_{J^w_d}] \sim \mathcal{N}(\vect{0}, \vect{I}_d)$ is independent of $\vect{W}_{\vect{J}}$. It follows that $\hat{\Rho}_0$ is \emph{unbiased} with 
\begin{align}\label{eq:xVecCov}
\Cov \hat{\Rho}_0 = \sigma^2  \Sx^{\frac{1}{2}} \E \left( (\vect{W}_{\vect{J}} \vect{W}_{\vect{J}}^T)^{-1} \right) \Sx^{\frac{1}{2}}.
\end{align}
In view of Lemma~\ref{lm:lemma1}, it is sufficient to analyze the performance of the unrealizable estimator $\hat{\Rho}_0$.
For the purpose of analyzing $\hat{\Rho}_0$ only, we can assume that Bob is given the value of $\vect{W}_{\vect{J}}$ for free, and the only cost is in transmitting the indices $J^w_1, \ldots, J^w_d$. Given $k_l$ bits for the representation of each of the locations, our general goal is to find $a,b$ that 
\begin{align}\label{eq:xVecOpt1}
&\textnormal{minimize } \tr \Cov \hat{\Rho}_0\\ \label{eq:xVecOpt2}
&\textnormal{subject to }  h_g(2Q(a)(1-2Q(b))^{d-1}) = k_l.
\end{align}
Before proceeding to the analysis of $\hat{\Rho}_0$, we need the following two technical lemmas.
\begin{lemma}\label{lm:matrices}
	Let $\vect{M}$ be a square random matrix with independent entries, where the diagonal entries are i.i.d. with one distribution, and the off-diagonal entries are i.i.d. with another, symmetric distribution. Then $\E \vect{M}$, $\E \vect{M} \vect{M}^T$ and $\E \left( (\vect{M} \vect{M}^T)^{-1} \right)$ are scalar multiples of the identity matrix.
\end{lemma}
\begin{proof}
	The claim for $\E \vect{M}$ and $\E \vect{M} \vect{M}^T$ is trivial. For $\E \left( (\vect{M} \vect{M}^T)^{-1} \right)$ see Appendix~\ref{apn:mtrcsProof}.
\end{proof}
\begin{lemma}[Johnson \cite{johnson1989gersgorin}] \label{lm:johnson}
	For any $n$-by-$m$ matrix $B = (b_{ij})$, $n \le m$, the smallest singular value is bounded below by
	\begin{align}
	\min_{i \in [n]} \left\lbrace  |b_{ii}| - \frac{1}{2} \left( \sum_{j \in [n] \setminus i} |b_{ij}| + \sum_{j \in [n] \setminus i} |b_{ji}|  \right) \right\rbrace
\end{align}
\end{lemma}
The following lemma provides a simplified expression and bounds for $\E \left( (\vect{W}_{\vect{J}} \vect{W}_{\vect{J}}^T)^{-1} \right)$, that will aid in proving Proposition~\ref{pr:prop1} below.
\begin{lemma}\label{lm:auxilarry} The following claims hold for 
	 \begin{align}
	&\alpha  = d^{-1}\tr \E \vect{W}_{\vect{J}} \vect{W}_{\vect{J}}^T, \;\;\beta = d^{-1}\tr \E \left( (\vect{W}_{\vect{J}} \vect{W}_{\vect{J}}^T)^{-1} \right). 
	 \end{align}
	\begin{enumerate}[label=(\roman*)]
		\item $\E \vect{W}_{\vect{J}} \vect{W}_{\vect{J}}^T = \alpha \vect{I}_d$ \label{lmi:itm1}
		\item $\E \left( (\vect{W}_{\vect{J}} \vect{W}_{\vect{J}}^T)^{-1} \right) = \beta \vect{I}_d$ \label{lmi:itm2}
		\item For any $a,b$ such that $a > (d-1)b$,
\begin{align}\label{eq:auxLemmBnds}
(a^2 + d + 1)^{-1} \le \alpha^{-1} \le \beta \le (a - (d-1)b)^{-2}.
\end{align}
\vspace{-10pt}
\label{lmi:itm3}
\end{enumerate}
\end{lemma}
\begin{proof}
	Recall that the vectors $\vect{W}_i$ are i.i.d. across the time index $i$, and that the entries of each one are i.i.d. with a symmetric distribution. Taking into account the rectangular structure of the stopping sets $A^w_\ell$ we see that $\vect{W}_{\vect{J}}$ has independent entries where diagonal elements have one distribution, and off-diagonal elements follow another, symmetric distribution. Thus, the matrix $\vect{W}_{\vect{J}}$ satisfies the conditions of Lemma~\ref{lm:matrices}. This proves claim~\ref{lmi:itm1}~(which also holds trivially by construction)~and claim~\ref{lmi:itm2}. 
	
	We proceed to prove claim~\ref{lmi:itm3}. Denoting the singular values of $\vect{W}_{\vect{J}}$ by $\sqrt{\lambda_1} \ge \ldots \ge \sqrt{\lambda_d}$, we have that 
		\begin{align}
		\tr (\vect{W}_{\vect{J}}\vect{W}_{\vect{J}}^T)^{-1}
		&=\sum_{\ell \in  [d]} \lambda_{\ell}^{-1} \le d\lambda_d^{-1}.
		\end{align}
		By construction, the diagonal entries of $\vect{W}_{\vect{J}}$ are larger than $a$ in absolute value, and the off-diagonal entries are smaller than $b$ in absolute value. Therefore Lemma~\ref{lm:johnson} yields
		\begin{align}
		\sqrt{\lambda_d} \ge a - (d-1)b. 
		\end{align}
		We thus have 
		\begin{align}\label{eq:frobBnd}
		\beta d &=  \tr \beta \vect{I}_d = \tr \E  (\vect{W}_{\vect{J}} \vect{W}_{\vect{J}}^T)^{-1} \le  d(a - (d-1)b)^{-2},
		\end{align}
		which establishes the rightmost inequality in claim~\ref{lmi:itm3}. The middle inequality holds since
		\begin{align}
		\beta d &=  \sum_{\ell \in [d]} \E  \frac{1}{\lambda_{\ell}} 
		\ge  \sum_{\ell \in [d]}   \frac{1}{\E \lambda_{\ell}} \ge  \frac{d^2}{\sum_{\ell \in [d]} \E \lambda_{\ell}} \\
		&= \frac{d^2}{\tr \E \vect{W}_{\vect{J}} \vect{W}_{\vect{J}}^T} = \frac{d^2}{d \alpha},
		\end{align}
		where the two inequalities follow from Jensen's inequality applied to the function $1/x$. Note that the rows and columns of $\vect{W}_{\vect{J}}$ have the same distribution. Therefore
		\begin{align}
		\alpha &= \E \|\vect{W}_{I_1}\|^2 \\
		&= \E \left( \left.  (\vect{W})_1^2 \right|  \left|(\vect{W})_1 \right| > a \right) + 
		(d-1)\E \left( \left. (\vect{W})_2^2 \right| \left|(\vect{W})_2\right| < b \right)\\
		&= 1 + as(a) + (d-1) \E \left( \left. (\vect{W})_2^2 \right| \left|(\vect{W})_2\right| < b \right)\\
		&\le 1 + a(a+a^{-1}) + (d-1)1 \\
		&= a^2 + 1 + d
		\end{align}
		which completes the proof.
\end{proof}
Lemma~\ref{lm:auxilarry} and \eqref{eq:xVecCov} implies that the optimization problem~\eqref{eq:xVecOpt1}-\eqref{eq:xVecOpt2} can be written as
\begin{align}
&\textnormal{minimize } \beta \\ \label{eq:abConstr1}
&\textnormal{subject to }  h_g(2Q(a)(1-2Q(b))^{d-1}) = k_l.
\end{align}
Note that both $Q(\cdot)$ and $h_g(\cdot)$ are monotonically decreasing. Therefore from \eqref{eq:abConstr1} it is clear that increasing $k_l$ means increasing $a$ and/or decreasing $b$. From \eqref{eq:auxLemmBnds} we get that $\beta$ decreases as $a$ increases \emph{and} gets farther away from $b$. We conclude therefore that a reasonable approximation to the solution of the optimization problem above, for large $k_l$, ia as given in~\eqref{eq:approxab}. Note that the proposed approximated solution satisfies the constraint exactly. 
\begin{proposition}\label{pr:prop1}
	The estimator \eqref{eq:unrlzblEst} is unbiased and, for the choice of $a,b$ given in~\eqref{eq:approxab}, it satisfies
	\begin{align}
	\tr \Cov \hat{\Rho}_0 \le  \frac{1}{k_l} \left(\frac{d }{2 \ln 2} \min_{\ell \in [d]} \{1 - \rho^2_{\ell}\} + o(1) \right) 
	\end{align}
	where $k_l$ is the expected number of bits used to describe \emph{each} of the locations $J^w_1, \ldots, J^w_d$. Furthermore, $\hat{\Rho}_0$ is asymptotically efficient given $(\vect{W}_\vect{J}, Y_\vect{J})$.
\end{proposition}
\begin{proof}
	In light of Lemma~\ref{lm:auxilarry}, \eqref{eq:xVecCov} can be written as 
	\begin{align}
	\Cov \hat{\Rho}_0 = \beta \sigma^2 \Sx. 
	\end{align}
	Using \eqref{eq:generalFisher}-\eqref{eq:generalFisher2} with $\mu =  \Rho \Sx^{-\frac{1}{2}} \vect{W}_{\vect{J}}$ and $\Sigma = \sigma^2 \mathbf{I}_d$, we get that the Fisher information matrix of $(\vect{W}_{\vect{J}}, Y_{\vect{J}})$ is 
	\begin{align}
	\fisher_{\vect{W}_{\vect{J}} Y_{\vect{J}}} &= \frac{1}{\sigma^2} \Sxsqm \E \vect{W}_{\vect{J}} \vect{W}_{\vect{J}}^T \Sxsqm + \frac{2d}{\sigma^4} \Sx^{-1} \Rho^T \Rho \Sx^{-1}\\
	&= \frac{\alpha}{\sigma^2} \Sx^{-1} + \frac{2d}{\sigma^4} \Sx^{-1} \Rho^T \Rho \Sx^{-1},
	\end{align}
	and using the Sherman–Morrison formula (e.g. \cite{hager1989updating}) we get
	\begin{align}
	\fisher_{\vect{W}_{\vect{J}} Y_{\vect{J}}}^{-1}  = \frac{\sigma^2}{\alpha}\left(\Sx - \frac{2 d }{\alpha \sigma^2 + 2 d  (1-\sigma^2)} \Rho^T \Rho \right).
	\end{align}
	We take $a,b$ of~\eqref{eq:approxab}. Note that $b$ is fixed and that $a$ increases with $k_l$.
	From Lemma~\ref{lm:auxilarry} we have that 
	\begin{align}
	&\alpha^{-1} \textbf{}=  a^{-2}(1+o(1)),\quad \beta =  a^{-2}(1+o(1)),
	\end{align}
	which implies 
	\begin{align}
	&\Cov \hat{\Rho}_0 =  \frac{\sigma^2}{a^2}(1+o(1)) \Sx\\
	&\fisher_{\vect{W}_{\vect{J}} Y_{\vect{J}}}^{-1} =  \frac{\sigma^2}{a^2}(1+o(1)) \Sx
	\end{align}
	and thus $\hat{\Rho}_0$ is asymptotically efficient. 
	For large $k_l$ we have 
	\begin{align}
	k_l &= h_g(Q(a)2(1-2Q(b))^{d-1}) \\
	&= -\log(Q(a)2(1-2Q(b))^{d-1})(1+o(1))\\
	&= -\log(Q(a))(1+o(1))\\
	&= \frac{a^2}{2 \ln 2}(1+o(1))\\
	\end{align}
	and thus 
	\begin{align}
	\tr \Cov \hat{\Rho}_0 &= d \beta \sigma^2 = \frac{d \sigma^2}{a^2}(1+o(1))\\
	&= \frac{1}{k_l} \left( \frac{d \sigma^2}{2 \ln 2}+o(1)  \right).
	\end{align}
	It remains to show that
	\begin{align}\label{eq:twotworbnd}
	\sigma^2 \le \min\{1 - \rho_{\ell}^2\}. 
	\end{align}
	Note that $\sigma^2 = \Var (Y|\vect{X})$ is the MMSE of estimating $Y$ from $\vect{X}$ (see e.g.~\cite{kay1993fundamentals}). Therefore it is not greater than  $1-\rho_{\ell}^2 = \Var (Y| (\vect{X})_{\ell})$, which is the MMSE of estimating $Y$ from the $\ell$-th coordinate only. 
\end{proof}
Theorem~\ref{thm:xvec} now follows from Lemma~\ref{lm:lemma1} and Proposition~\ref{pr:prop1}.

\section{Non-Gaussian Families} \label{s:Apps}
In this section, we move beyond the Gaussian setup and consider the problem of distributed correlation estimation in more general families of distributions, based on our Gaussian constructions. For brevity of exposition, we limit our discussion to the scalar case; the results can be extended in an obvious way to the vector case. We note that in contrast to the Gaussian setting, the marginal distribution of $X$ or $Y$ in other families of distributions may depend on the correlation, in which case Alice or Bob could use their (unlimited) local measurements to improve their inference (and in some cases to even learn $\rho$ exactly without any communication). For example, if $X$ is uniformly distributed over the interval $[-\sqrt{3},\sqrt{3}]$, and $Y = \rho X+ \sqrt{1-\rho^2}Z$ where $Z$ is uniformly distributed over the discrete set $\{-1,1\}$, then it is clear that the distribution of $Y$, which can be determined with arbitrary accuracy by Bob, determines $\rho$ up to its sign, reducing our problem to a binary hypothesis testing one. Such scenarios render our method useless, or, at the very least, degenerate. 

Our interest, therefore, is in families of distributions where the marginals reveal little or nothing about the correlation. Specifically, we say that a family $\mathscr{F}$ of distributions on $(X,Y)$  is \emph{correlation-hiding} if each pair of marginals can be associated with an infinite number of possible correlations; namely, for any two marginals $p_X$ and $p_Y$ that are possible for some member of $\mathscr{F}$, there exists a countably infinite set $\mathscr{F}'\subseteq \mathscr{F}$ of joint distributions with marginals $p_X$ and $p_Y$, and with correlation coefficients that are all distinct. 

Below, we discuss two types of correlation-hiding families. The first is the family of all possible distributions (subject only to mild moment constraints), which is obviously correlation-hiding. We show that for this family, the Gaussian performance can be uniformly attained. The idea is very simple: we perform ``Gaussianization'' of the samples using the Central Limit Theorem (CLT), and then apply the Gaussian estimators; showing that this indeed works, however, is somewhat technically involved. The second type of families that we consider are ones where $p_X$ is known, and where $Y=\alpha X+Z$ for some unknown coefficient $\alpha$ and unknown independent noise $Z$. We show that such families are correlation-hiding, and that we can sometimes (depending on $p_X$) obtain a variance that decays much faster with $k$ than the Gaussian one. 

\subsection{Unknown Distributions}\label{ss:unknwnDists}
In this subsection, we consider the case where the joint distribution of $X$ and $Y$ is completely unknown, subject only to mild moment conditions. We show how the threshold method of Subsection~\ref{ss:sclrThresh} can be extended to this setup, using the CLT, to yield the same performance guarantees. The basic idea is to use the unlimited number of samples in order to create Gaussian r.v.s with the same correlation, by averaging over blocks of samples. Due to the CLT, it is intuitively clear that this approach works if Alice and Bob use infinite sized blocks. This is however impractical, and the main technical challenge is to show that using finite large enough blocks, i.e., changing the order of limits, still works. 

Let $(X,Y)$ be drawn from the family
\begin{align}\label{eq:general_family_def}
	\mathscr{F} = \{p_{XY} : \E X^2, \E Y^2 < u, \E Y^4<\infty\}. 
\end{align}
where $u$ is some known constant. Again, since we assume that local measurements are essentially unlimited, and the second moments have known upper bounds, we can assume without loss of generality that $\E X = \E Y = 0$, and $\E X^2 = \E Y^2 = 1$. The following claim is immediate from the fact that $\mathscr{F}$ contains in particular the Gaussian distributions.  
\begin{corollary}
	The family $\mathscr{F}$ in~\eqref{eq:general_family_def} is correlation-hiding. 
\end{corollary}

Let us now proceed to describe our estimator. Alice and Bob first locally sum over their measurements to create the new i.i.d. sequences $\{\bar{X}_i \}_i, \{\bar{Y}_i \}_i$, given by
\begin{align}
&\bar{X}_i = \frac{1}{\sqrt{m}} \sum_{j \in S_i} X_j, \quad \bar{Y}_i = \frac{1}{\sqrt{m}} \sum_{j \in S_i} Y_j
\end{align}
where the $S_i$'s are disjoint index sets of size $m$. For brevity, we suppress the dependence of these new r.v.s on $m$. The sequence of pairs $\{(\bar{X}_i, \bar{Y}_i)\}_i$ is clearly i.i.d. Denoting by $(\bar{X},\bar{Y})$ a generic pair in this sequence, the correlation between $\bar{X}$ and $\bar{Y}$ is clearly the same as the correlation between $X$ and $Y$. Alice and Bob can therefore apply the threshold method to the sequence $\{(\bar{X}_i, \bar{Y}_i)\}_i$ in order to estimate the original $\rho$. We now show that the performance of this estimator approaches the Gaussian performance as $m \rightarrow \infty$. Given a communication constraint of $k$ bits, the threshold $t$ is chosen (as in the Gaussian case) such that  $h_g(Q(t)) = k$.  We denote
\begin{align}
\bar{J} = \min \{i: \bar{X}_i > t\}
\end{align}
and the estimator
\begin{align}
\hat{\rho}_{\th}^{(m)} = \frac{\bar{Y}_{\bar{J}}}{s(t)},
\end{align}
where $s(t)$ is given in \eqref{eq:invMillsDef}. Note we cannot normalize by $\E \bar{X}_{\bar{J}}$ to get a strictly unbiased estimator since we assume unknown distributions and thus $\E \bar{X}_{\bar{J}}$ is not known for finite $m$.
The expected number of bits needed to describe $\bar{J}$ is
\begin{align}
k^{(m)} = h_g(\Pr(\bar{X} > t)).
\end{align}
%The notations of the equivalent Gaussian quantities (namely $X, Y, J,\hat{\rho}_{\th}$ etc.) are as in Subsection~\ref{ss:sclrThresh}, with the additional bar. 
\begin{remark}
	Note that practical scenarios would require the choice of some fixed $m$. Therefore, in cases where the support of $X$ is finite, we might get that $\Pr(\bar{X} > t) = 0$ which means Alice waits forever and the estimator is undefined. Therefore, while the distribution of $(X, Y)$ need not be known in general, such a practical scenario requires some knowledge regarding the support of $X$ in the form of a number $x$ such that $\Pr (X_i > x)>0$ (which must exist since $\E X = 0$). Then we can take $m > t^2/x^2$ to assure $\Pr(\bar{X} > t) > 0$.
\end{remark}
\begin{theorem} \label{thm:CLTgenrlz} Let $t = Q^{-1}(h_g^{-1}(k))$. Then for the family $\mathscr{F}$ in~\eqref{eq:general_family_def} it holds that $\lim_{m \rightarrow \infty} k^{(m)} = k$ and
	\begin{align}
	\lim_{m \rightarrow \infty} \E(\hat{\rho}_{\th}^{(m)} - \rho)^2 = \frac{1}{k} \left( \frac{1 - \rho^2}{ 2 \ln 2} + o(1)  \right)
	\end{align}
\end{theorem}
\begin{proof}
	Due to the CLT we have for any fixed $t>0$ that
	\begin{align}\label{eq:cltlim}
	\lim_{m \rightarrow \infty} \Pr(\bar{X} > t) = Q(t)
	\end{align}
	and thus, since $h_g$ is smooth, the communication constraint is asymptotically satisfied. We have
	\begin{align}
	&\E(\hat{\rho}_{\th}^{(m)} - \rho)^2 =  \frac{\E  \bar{Y}_{\bar{J}}^2  }{s^2(t)} - 2\rho  \frac{ \E \bar{Y}_{\bar{J}}}{s(t)} + \rho^2 
%	&\\ \E(\hat{\rho}_{\th} - \rho)^2 =  \frac{\E Y_J^2  }{s^2(t)} - 2\rho  \frac{ \E Y_J}{s(t)} + \rho^2
	\end{align}
	and thus it suffices to show that the first two moments of $\bar{Y}_{\bar{J}}$ converge to their values under the Gaussian distribution. Denoting by $(X^{\mathcal{N}},Y^{\mathcal{N}})$ and $Y^{\mathcal{N}}_{J}$ the associated r.v.s under a Gaussian distribution, it is enough to show that $\bar{Y}_{\bar{J}}$ converges in distribution to $Y^{\mathcal{N}}_{J}$ as $m\to \infty$, and that $\bar{Y}_{\bar{J}}^2$ is uniformly integrable \cite{billingsley1995probability}. To show convergence in distribution, observe that
	\begin{align}
	\lim_{m \rightarrow \infty} \Pr (\bar{Y}_{\bar{J}} > y) 
	&=\lim_{m \rightarrow \infty} \Pr (\bar{Y} > y| \bar{X} > t) \\
	&=\lim_{m \rightarrow \infty} \frac{\Pr (\bar{Y} > y, \bar{X} > t) }{\Pr(\bar{X} > t)} \\ \label{eq:genTrns1}
	&= \frac{\displaystyle \lim_{m \rightarrow \infty} \Pr (\bar{Y} > y, \bar{X} > t) }{\displaystyle \lim_{m \rightarrow \infty} \Pr(\bar{X} > t)} \\ \label{eq:genTrns2}
	&= \frac{ \Pr (Y^{\mathcal{N}} > y, X^{\mathcal{N}} > t) }{ \Pr(X^{\mathcal{N}} > t)} \\
	&= \Pr (Y^{\mathcal{N}} > y| X^{\mathcal{N}} > t) \\
	&= \Pr (Y^{\mathcal{N}}_{J} > y)
	\end{align}
	where \eqref{eq:genTrns1} holds since the denominator is not zero, and \eqref{eq:genTrns2} holds by virtue of the CLT.
	
	It follows from \eqref{eq:cltlim} that there exist some $m_0$ and $c>0$ (e.g., $c = Q(t)/2$) such that 
	\begin{align}
	\Pr(\bar{X} > t) > c \ \ \ \forall \ m \ge m_0,
	\end{align}
	and therefore we assume without loss of generality that $m\ge m_0$. To prove uniform integrability of $\bar{Y}_{\bar{J}}^2$ it suffices to show that $\sup_m \E |\bar{Y}_{\bar{J}}|^{\gamma} < \infty$ for some $\gamma>2$ \cite{billingsley1995probability}. For simplicity, we set $\gamma=4$:
	\begin{align}
	\E |\bar{Y}_{\bar{J}}|^4 &= \E (|\bar{Y}|^4 \mid  \bar{X} > t) \\
	&\le \frac{\E |\bar{Y}|^4 }{\Pr(\bar{X} > t)}\\
	&= \frac{\E (\frac{1}{\sqrt{m}} \sum_j Y_j)^4 }{\Pr(\bar{X} > t)}\\
	&= \frac{ \frac{1}{m} \E Y^4 + 3 \frac{m-1}{m} (\E Y^2)^2}{\Pr(\bar{X} > t)} \\
	&< \frac{1}{c} \left(\frac{\E Y^4 }{m}  + 3\right), 
	\end{align}
	which is finite since $\E Y^4< \infty$.
\end{proof}
\begin{example}[Doubly symmetric binary r.v.s]
	Consider the family of distributions where $X\sim \mathrm{Bernoulli}(1/2)$ and $Y=X\oplus Z$ where $Z\sim \mathrm{Bernoulli}(p)$ is independent of $X$, $p\in[0,1]$ is unknown, and $\oplus$ is the binary XOR operation. The associated Gaussian version of these r.v.s (after removing the mean) are the jointly normal, zero mean unit norm r.v.s $\bar{X}$ and $\bar{Y}$, with correlation $\rho=1-2p$. Our unbiased estimator can therefore obtain a variance of $\frac{1-(1-2p)^2}{2k\ln 2} = \frac{2p(1-p)}{k\ln{2}}$ for the estimation of $\rho$, which corresponds to a variance of $\frac{p(1-p)}{2k\ln{2}}$ for the estimation of $p$. This can be juxtaposed with the straightforward approach of simply sending $X_1,\ldots,X_k$ to Bob and applying the (efficient) estimator $\hat{p} = \frac{1}{k}\sum_{j=1}^k X_j\oplus Y_j$. This unbiased estimator has a variance of $\frac{p(1-p)}{k}$, which is interestingly slightly worse than what we got using the Gaussian approach. It may be possible to improve the former by using lossy compression, but we do not explore this direction here. 
\end{example}
\begin{remark}
Estimating the joint probability mass function of general discrete distributions on $X,Y$ can be similarly cast as a correlation estimation problem. However, the gain observed in the binary case above does not carry over to the general case. This is however not unexpected, since our estimator does not assume any bound on the cardinality of $X$ and $Y$. 
\end{remark}

\subsection{Additive Noise Families}\label{ss:genAdditive}
In this subsection, we consider a more restricted model where the distribution $p_X$ of $X$ is fixed (but not necessarily Gaussian) and has bounded variance, and where 
\begin{align}\label{eq:addtveChnlModel}
Y = \alpha X + Z 
\end{align}
for some unknown bounded constant $\alpha$, where $Z$ is an arbitrary r.v. with bounded variance that is independent\footnote{It is in fact sufficient for our purposes to assume only that $\E(Z | X)$ and $\Var(Z | X)$ do not depend on $X$} of $X$. Let us denote this family of distributions by $\mathscr{F}(p_X)$. First, we note:
\begin{corollary}\label{cor:cor-hid2}
	$\mathscr{F}(p_X)$ is correlation-hiding for any $p_X$.  
\end{corollary}
\begin{proof}
	See Appendix~\ref{apn:cor-hid2-proof}.
\end{proof}

We now show that the threshold estimator proposed for the Gaussian case applies to $\mathscr{F}(p_X)$ as well, and that its performance can be better or worse, depending on $p_X$. Specifically, we show that the $O(1/k)$ decay of the variance with the number of bits is not fundamental, as for some (heavier tailed) choices of $p_X$ we obtain a behavior of $O(1/k^2)$ using the same threshold estimator, and $2^{-\Omega(k)}$ using a slightly modified estimator. The latter is essentially the best possible using our approach, since we utilize $O(2^k)$ samples (with high probability), which corresponds to a variance of $\Omega(2^{-k})$ even in the centralized case. 

As in the previous sections, Alice and Bob can normalize their measurements locally. Therefore, we can assume without loss of generality that~\eqref{eq:addtveChnlModel} can be written as
\begin{align}\label{eq:genAddtvModel}
Y = \rho X + \sqrt{1 - \rho^2} Z
\end{align}
where $X$ and $Z$ are independent, zero mean unit variance r.v.s, and the correlation is $\rho = \E XY$. We assume that $p_Z$ is arbitrary and unknown, and that $p_X$ is arbitrary but known. Applying the threshold method of Subsection~\ref{ss:sclrThresh} to this non-Gaussian setup, we denote as usual $J = \min \{i: X_i > t\}$ the first index to pass the threshold $t$, where $t$ is chosen such that $h_g (\Pr (X > t)) = k$. Our estimator is 
\begin{align}
\hat{\rho}_{\th} = \frac{Y_J}{\E X_J}.
\end{align}
The following claim is immediate. 
\begin{corollary}
	$\hat{\rho}_{\th}$ is unbiased, and 
	\begin{align}\label{eq:addChnlVar}
	\Var \hat{\rho}_{\th} 	&= \frac{\rho^2 \Var (X \mid X > t) + 1 - \rho^2}{(\E (X \mid X > t))^2}.
	\end{align}
\end{corollary}
Let us compute ~\eqref{eq:addChnlVar} for some specific choices of $p_X$.
\begin{example}[Laplace Distribution]
Let $p_X$ be a zero-mean, unit-variance Laplace distribution, hence $\Pr (X > x) = \frac{1}{2} e^{-\sqrt{2} x}$ for $x>0$. Thus,
\begin{align}
&\E (X \mid X >t) = t + \frac{1}{\sqrt{2}}, \quad \Var (X \mid X > t) = \frac{1}{2},
\end{align}
and
\begin{align}
k 
&= h_g \left( \frac{1}{2} e^{-\sqrt{2} t} \right)\\
&= - \log \left( \frac{1}{2} e^{-\sqrt{2} t} \right) (1 + o(1))  \\
&= \sqrt{2} t \log (e) (1 + o(1)).
\end{align}
Therefore \eqref{eq:addChnlVar} becomes
\begin{align}
\Var \hat{\rho}_{\th} 
= \frac{ 1}{  k^2} \left( \frac{2 -  \rho^2}{(\ln 2)^2}   +o(1) \right),
\end{align}
which yields a variance of $O(1/k^2)$, in contrast to the slower $O(1/k)$ attained in the Gaussian case.
\end{example}

\begin{example}[Pareto Distribution]
Motivated by the Laplace example which indicates that a heavier tail of $X$ may yield a faster decay of $\Var \hat{\rho}_{\th}$, we investigate the heaviest tail possible with finite variance. Suppose that $p_X$ is the (double-sided, zero mean) Pareto distribution, i.e.,   
\begin{align}\label{eq:paertoCDF}
\Pr (X > x) = \Pr (X < -x) = \frac{1}{2} \left( \frac{x_0}{x} \right)^{\alpha}
\end{align}
for any $x > x_0$, where $\alpha > 2$ and $x_0 > 0$ is set such that $\Var X = 1$. Then for any $t>x_0$
\begin{align}
\E (X \mid X >t) \!=\! \frac{\alpha t}{\alpha - 1},  \quad \Var (X \mid X > t) \!=\! \frac{\alpha t^2}{(\alpha - 1)^2 (\alpha - 2)}
\end{align}
and \eqref{eq:addChnlVar} becomes
\begin{align}
\Var \hat{\rho}_{\th} = \frac{\rho^2}{\alpha (\alpha - 2)} + O(1/t^2).
\end{align}
Thus, the variance of our threshold estimator does not vanish with the number of bits. This flaw can nevertheless be fixed in a very strong way, as we show next. Before we proceed, we note that for $p_X$ with a tail of the form $\Pr (X>x) \propto e^{-x^{\frac{1}{m}}}$, i.e. in between Pareto and Laplace, the threshold estimator yields $\Var \hat{\rho}_{\th} = O(1/k^2)$ for any natural $m$. Also, tails that decay faster than Gaussian may yield worse performance, e.g. the tail $e^{-x^4}$  yields  $\Var \hat{\rho}_{\th} = O(1/\sqrt{k})$. 
\end{example}

Getting back to the double-sided Pareto distribution, recall that in the Gaussian case it was shown that describing the value of $X_J$ does not improve estimation performance. This was due to the fact that for the Gaussian family, $\Var X_J = \Var (X \mid X > t) \rightarrow 0$. This is however not true in general; in fact, the Pareto distribution is an extreme case in which  $\Var (X \mid X > t) \rightarrow \infty$. Therefore, providing some information regarding the value of $X_J$ at the expense of the number of bits used to describe the index $J$, might improve performance. With that in mind, we consider the estimator
\begin{align}\label{eq:thrshldestQntzd}
\hat{\rho}_{\thq} = \frac{Y_J}{\hat{X}_J}
\end{align}
that allocates $k_l$ bits to describe $J$, and $k_q$ bits to describe the value of $\hat{X}_J$, where $k_l + k_q = k$. We apply the following simple quantizer. For some $u>t$ we divide the region $[t, u]$ to $2^{k_q}$ equal segments of length $\Delta = 2^{-k_q}(u-t)$. For $x>u$ we set $\hat{x} = u$. In the following, we show that this estimator attains a variance that decays exponentially fast with $k$. 
\begin{proposition} \label{pr:paretoQntz} 
	Consider the family $\mathscr{F}(p_X)$ where $p_X$ be the double-sided Pareto distribution. Then the estimator $\hat{\rho}_{\thq}$ in~\eqref{eq:thrshldestQntzd} satisfies
	\begin{align}
	\E (\hat{\rho}_{\thq} - \rho)^2 
	\le  (1 + \rho^2)\cdot 2^{-\frac{2 }{\alpha}  \frac{\alpha-2}{\alpha-1}k(1 - o(1))},
	\end{align}	
	where $k$ is the average number of transmitted bits.
\end{proposition}
\begin{proof}
	See Appendix~\ref{apn:proofParetoProp}.
\end{proof}

\section{Conclusions}
We have discussed the problem of estimating the correlations between remotely observed random vectors with unlimited local samples, under one-way communication constraints. For the case where the vectors are jointly Gaussian, we provided simple constructive unbiased estimators for the correlations; our estimators attain the best known non-constructive Zhang-Berger upper bound on the variance in the scalar case, and use the local correlations to uniformly improve performance in the vector case, where the Zhang-Berger approach seems inapplicable. Loosely speaking, our approach is based on Alice scanning her the local observations and sending the index of suitably ``large'' samples that induce good signal-to-noise ratio for the estimation for Bob, who uses the corresponding samples on his end. We then showed that using the CLT, this approach can be applied to the case of estimating correlations for completely unknown distributions, with the exact same variance guarantees. While the Gaussian approach yields a variance that is inversely proportional to the expected number of transmitted bits, we show that for joint distributions generated via unknown fading channels with unknown additive noise, whose correlations cannot be estimated locally, a slightly modified estimator attains a variance decaying {\em exponentially fast} with the expected number of transmitted bits. It remains interesting to try and obtain lower bounds on the variance as a function of the number of bits and the richness of the family of distributions under consideration. We conjecture that the inversely proportional behavior of our Gaussian estimator is order-wise optimal in the Gaussian case, hence also for the case of unknown distributions. 

\appendix
\section{Appendix}
\subsection{Proof of Theorem~\ref{thm:yvec}}\label{apn:yVecthmProof}
The model can be written as (see e.g. \cite{rencher2003methods})
\begin{align}
\vect{Y} = \Rho X + \Sigma^{\frac{1}{2}} \vect{Z}
\end{align}
where $\vect{Z} \sim \mathcal{N}(\vect{0},\vect{I}_d)$ is independent of $X$, and $\Sigma = \Sigma_{\vect{Y}} - \Rho \Rho^T$.
We have $\hat{\Rho} = (\Rho X_J +  \Sigma^{\frac{1}{2}} \vect{Z}_J)/\E X_J$. Therefore $\E \hat{\Rho} = \Rho$ and 
\begin{align}
\Cov \hat{\Rho} =  \frac{1}{(\E X_J)^2} \left(\Sigma +  \Var (X_J) \Rho \Rho^T \right).
\end{align}
Using \eqref{eq:generalFisher}-\eqref{eq:generalFisher2} with $\mu = \Rho X_J, \Sigma = \Sigma_{\vect{Y}} - \Rho \Rho^T$ we get that the Fisher information matrix pertaining to $(X_J,\vect{Y}_J)$ is 
\begin{align}
\fisher_{X_J \vect{Y}_J} = \Sigma^{-1}(\E X_J^2 + \Rho^T \Sigma^{-1} \Rho ) + \Sigma^{-1} \Rho \Rho^T \Sigma^{-1},
\end{align}
and applying the Sherman–Morrison formula (e.g. \cite{hager1989updating}) yields
\begin{align}
\fisher_{X_J \vect{Y}_J}^{-1} = \frac{1}{EX_J^2 + \Rho^T \Sigma^{-1} \Rho} \left(\Sigma - \frac{\Rho \Rho^T}{EX_J^2 + 2\Rho^T \Sigma^{-1} \Rho}  \right).
\end{align}
Using the arguments of Theorem~\ref{thrm:stEst} yields that both $\Cov \hat{\Rho}$ and  $\fisher_{X_J \vect{Y}_J}^{-1}$ are $\frac{1}{t^2} (1 + o(1) ) \Sigma$ and thus $\hat{\Rho}$ is asymptotically efficient. 
Theorem~\ref{thrm:stEst} also implies that $k = t^2 (2\ln 2 + o(1))$, and noting that $\tr \Sigma  = \tr (\Sigma_{\vect{Y}} - \Rho \Rho^T) = d - \|\Rho\|^2$ concludes the proof.

%\subsection{Optimal scalar transformation is not uniformly better than `naive' scalar method}\label{apn:SclrLinTrnsfsm}
\subsection{The scalar method with linear transformations}\label{apn:SclrLinTrnsfsm}
In this subsection we show that  any  method based on $d$ scalar transmissions cannot uniformly beat the scheme of applying the basic scalar method $d$ times. In this sense, the joint method proposed in Theorem~\ref{thm:xvec} \emph{is superior} because it \emph{does} uniformly beat the simple scalar scheme. Specifically, let $M$ be some invertible $d \times d$ matrix known to both Alice and Bob, and let $\widetilde{\vect{X}} = M\vect{X}$. Suppose Alice and Bob apply the scalar method separately to obtain an estimator $\hat{\Rho}_M$ for the correlation vector $\Rho_M = \E Y \widetilde{\vect{X}}^T$, and then use the estimator $M^{-1}\hat{\Rho}_M$ to estimate $\Rho$. As it turns out, this family of estimators does not dominate the naive approach of estimating each correlation separately (i.e., $M= \vect{I}_d$).  
\begin{proposition}\label{pr:sclarNoGood}
	For any two invertible $d\times d$ matrices $M_1,M_2$ (that can arbitrarily depend on $\Sx$, and are known to both Alice and Bob), $M_1^{-1}\hat{\Rho}_{M_1}$ does not dominate $M_2^{-1}\hat{\Rho}_{M_2}$. 
	%Let $M_1$ and $M_2$ be any two invertible $d\times d$ matrices that can arbitrarily depend on $\Sx$, and are known to both Alice and Bob. Then $M_1^{-1}\hat{\Rho}_{M_1}$ does not dominate $M_2^{-1}\hat{\Rho}_{M_2}$. 
\end{proposition}
\begin{proof}
We need to show that any linear transformation applied to $\vect{X}$, followed by the scalar method, cannot be uniformly better than the scalar method itself. It suffices to show that for the two-dimensional case. 

Alice creates the following two scalar sequences.
\begin{align}
&U_i = [a_1, b_1] \vect{X}_i, \ \text{for} \ i = 1, \ldots, n_1 \ \text{and} \\
&V_i = [b_2, a_2] \vect{X}_i, \ \text{for} \ i = n_1+1, \ldots, n_1 + n_2
\end{align}
and allocates $k_1$ bits for $U$, and $k_2$ bits for $V$ (Note we can use either max or threshold method, and that $n_1,n_2$ can be arbitrarily large).
One special case of the above is the ``successive refinement'' approach described in the introduction (for $b_1 = 0$), and another special case is  the naive scalar method (for $a_1 = a_2 = 1$, $b_1 = b_2 = 0$ and $k_1 = k_2 = k/2$). Without loss of generality we assume $a_1, b_1$ are such that $\E U^2 = 1$, and $a_2, b_2$ are such that $\E V^2 = 1$. We denote
\begin{align}
&\alpha_1 = \E Y_i U_i = a_1 \rho_1 + b_1 \rho_2 \\
&\alpha_2 = \E Y_i V_i = b_2 \rho_1 + a_2 \rho_2,
\end{align}
and $\alphaBold = [\alpha_1, \alpha_2]^T$. We also denote $\Rho = [\rho_1, \rho_2]^T$ and
\begin{align}
M = \left[
\begin{matrix}
a_1 & b_1 \\
b_2 & a_2 
\end{matrix}
\right],
\end{align}
and therefore we have $\alphaBold = M \Rho$.
The best Bob can do (recall $U,V$ are independent) is to estimate $\alpha_1$ using $U$ and $\alpha_2$ using $V$ to obtain
\begin{align}
&\Var \hat{\alpha}_1 = \frac{1}{k_1} \left( \frac{1 - \alpha_1^2}{ 2 \ln 2} + o(1) \right)  \\
&\Var \hat{\alpha}_2 = \frac{1}{k_2} \left( \frac{1 - \alpha_2^2}{ 2 \ln 2} + o(1) \right) 
\end{align}
and then take $\hat{\Rho}_{\textnormal{trn}} = M^{-1} \hat{\alphaBold}$. The resulting sum of variances (note $\Cov(\hat{\alpha}_1, \hat{\alpha}_2)=0$)  is
\begin{align} 
\tr \Cov \hat{\Rho}_{\textnormal{trn}}  &= \tr  M^{-1} \Cov(\hat{\alpha})  M^{-T} \\
&= \tr  M^{-1} 
\left[
\begin{matrix}
\Var \hat{\alpha}_1 & 0 \\ 
0 & \Var \hat{\alpha}_2 
\end{matrix}
\right]
M^{-T}\\ 
&= \frac{(a_2^2 + b_2^2) \Var \hat{\alpha}_1 + (a_1^2 + b_1^2) \Var \hat{\alpha}_2 }{(a_1 a_2 - b_1 b_2)^2}
%&\sim \frac{1}{2 \ln 2} \frac{(a_2^2 + b_2^2) \frac{1 - \alpha_1^2}{k_1} + (a_1^2 + b_1^2) \frac{1 - \alpha_2^2}{k_2}}{(a_1 a_2 - b_1 b_2)^2}
\end{align}
Applying the simple scalar method twice yields 
\begin{align}\label{eq:varScalarTwice}
\tr \Cov \hat{\Rho}_{\textnormal{scl}}  = \frac{1}{k} \left( 
\frac{k}{k'_1}  \frac{1 - \rho_1^2}{ 2 \ln 2} +  \frac{k}{k'_2}\frac{1 - \rho_2^2}{ 2 \ln 2} + o(1) 
\right).
\end{align}
with $k'_1 + k'_2 = k_1 + k_2 = k$.
We want to show that $\tr \Cov \hat{\Rho}_{\textnormal{trn}}$ cannot be uniformly better than $\tr \Cov \hat{\Rho}_{\textnormal{scl}}$, namely, show that for any choice of $a_1,b_1,a_2,b_2,k_1,k_2$ (that do not depend on $\rho_1, \rho_2$) we can find $\rho_1, \rho_2$ such that $\tr \Cov \hat{\Rho}_{\textnormal{scl}} < \tr \Cov \hat{\Rho}_{\textnormal{trn}}$. This is easy because we can always take $\rho_1,\rho_2 \in \{-1,1\}$ (or arbitrarily close to $\pm 1$) which makes $\tr \Cov \hat{\Rho}_{\textnormal{scl}} \approx 0$ and $\tr \Cov \hat{\Rho}_{\textnormal{trn}} \neq 0$. If $\tr \Cov \hat{\Rho}_{\textnormal{trn}} = 0$ (i.e. $\alpha_1^2=\alpha_2^2=1$), we can flip the sign of $\rho_2$ to obtain either $\alpha_1^2 \neq 1$ or $\alpha_2^2 \neq 1$.
\end{proof}

\subsection{Proof of Lemma~\ref{lm:lemma1}}\label{apn:lm1Proof}
Writing $\vect{W} = \vect{W}_{\vect{J}}$ and $\hat{\vect{W}} = \hat{\vect{W}}_{\vect{J}}$, we have
\begin{align}
&\hat{\Rho}_0 = Y_{\vect{J}} \vect{W}^{-1} \Sxsq = \Rho + \sigma Z_{\vect{J}} \vect{W}^{-1} \Sxsq\\
&\hat{\Rho} = Y_{\vect{J}} \hat{\vect{W}}^{-1} \Sxsq \!=\! \Rho \Sxsqm \vect{W} \hat{\vect{W}}^{-1} \Sxsq  \!+\! \sigma Z_{\vect{J}} \hat{\vect{W}}^{-1} \Sxsq\\
&\hat{\Rho}_0 - \Rho  = \sigma Z_{\vect{J}} \vect{W}^{-1} \Sxsq\\
&\hat{\Rho} - \hat{\Rho}_0 = \Rho \Sxsqm (\vect{W} \hat{\vect{W}}^{-1} - I) \Sxsq + \sigma Z_{\vect{J}} (\hat{\vect{W}}^{-1} - \vect{W}^{-1}) \Sxsq.
\end{align}
Recall $Z_{\vect{J}}$ is a row vector $\sim \mathcal{N}(0,\mathbf{I}_d)$ independent of $\vect{W}$. It follows that
\begin{align}
&\E \|\hat{\Rho} - \hat{\Rho}_0\|^2 = \E \|\Rho \Sxsqm (\vect{W} \hat{\vect{W}}^{-1} - I) \Sxsq \|^2  \\
& \ +\sigma ^2 \E \tr (\hat{\vect{W}}^{-1} - \vect{W}^{-1}) \Sx (\hat{\vect{W}}^{-T} - \vect{W}^{-T}), 
\end{align}
and
\begin{align}
&\E(\hat{\Rho} - \hat{\Rho}_0)(\hat{\Rho}_0 - \Rho)^T = \sigma ^2 \E \tr (\hat{\vect{W}}^{-1} - \vect{W}^{-1})\Sx    \vect{W}^{-T}.
\end{align}
Therefore
\begin{align}
&\E \|\hat{\Rho} - \Rho\|^2 = \E \|(\hat{\Rho} - \hat{\Rho}_0) + (\hat{\Rho}_0 - \Rho)\|^2\\
&= \E\|\hat{\Rho}_0 - \Rho\|^2 
+\E \|\hat{\Rho} - \hat{\Rho}_0\|^2  
+ 2\E(\hat{\Rho} - \hat{\Rho}_0)(\hat{\Rho}_0 - \Rho)^T\\
&= \E\|\hat{\Rho}_0 - \Rho\|^2  +\E \|\Rho \Sxsqm (\vect{W} \hat{\vect{W}}^{-1} - I) \Sxsq \|^2 \\
& \ +\sigma ^2 \E \tr (\hat{\vect{W}}^{-1} - \vect{W}^{-1}) \Sx (\hat{\vect{W}}^{-T} + \vect{W}^{-T}),
\end{align}
and thus
\begin{align}
&\E\|\hat{\Rho} - \Rho\|^2 -  \E\|\hat{\Rho}_0 - \Rho\|^2\\ \label{eq:trm1}
&= \E \|\Rho \Sxsqm (\vect{W} - \hat{\vect{W}})\hat{\vect{W}}^{-1} \Sxsq \|^2\\ \label{eq:trm2}
&\  + \sigma ^2 \E \tr (\vect{W} - \hat{\vect{W}})\hat{\vect{W}}^{-1} \Sx (\hat{\vect{W}}^{-T} + \vect{W}^{-T}) \vect{W}^{-1}.
\end{align}

Let us upper bound the two terms separately. Recall that by~\eqref{eq:frobBnd} we have $\|\vect{W}^{-1}\|_F^2 \le d/(a - (d-1)b)^2$, which also holds for $\hat{\vect{W}}^{-1}$. Furthermore, the assumption that $a > d(b + 1)$ implies that $a - (d-1)b$ is lower bounded by either $a/d$ or $d$.

First term \eqref{eq:trm1}: The Frobenius norm is sub-multiplicative (see e.g. \cite{horn1990matrix}), and therefore
\begin{align}
&\E \|\Rho \Sxsqm (\vect{W} - \hat{\vect{W}})\hat{\vect{W}}^{-1} \Sxsq \|_F^2\\ \label{eq:ftTrn1}
&\le \|\Rho \Sxsqm\|_F^2  \|\Sxsq \|_F^2  \E \|\vect{W} - \hat{\vect{W}}\|_F^2 \|\hat{\vect{W}}^{-1}\|_F^2 \\ \label{eq:ftTrn2}
&\le \frac{d \|\Rho \Sxsqm\|_F^2  \|\Sxsq \|_F^2}{(a - (d-1)b)^2}   \E \|\vect{W} - \hat{\vect{W}}\|_F^2\\ 
&\le \frac{d \|\Rho \Sxsqm\|_F^2  \|\Sxsq \|_F^2}{(a/d)^2}   \E \|\vect{W} - \hat{\vect{W}}\|_F^2\\ \label{eq:ftTrn3}
&= \frac{d^4 (1 - \sigma^2) }{a^2}   \E \|\vect{W} - \hat{\vect{W}}\|_F^2
\end{align}
 where for~\eqref{eq:ftTrn3} we used the fact that $\|\Rho \Sxsqm\|^2 = \Rho \Sx^{-1} \Rho^T = 1 - \sigma^2$, and $ \|\Sxsq \|_F^2 = \tr \Sx = d$.

Second term \eqref{eq:trm2}: For any two $d \times d$ matrices $A,B$, it can be easily shown that $\tr AB ^T \le d^2 \|A\|_F \|B\|_F$. Therefore, 
\begin{align}
&\sigma ^2 \E \tr (\vect{W} - \hat{\vect{W}})\hat{\vect{W}}^{-1} \Sx (\hat{\vect{W}}^{-T} + \vect{W}^{-T}) \vect{W}^{-1}\\
&\le \sigma ^2 d^2  \E  \|\vect{W} - \hat{\vect{W}}\|_F \|\hat{\vect{W}}^{-1} \Sx (\hat{\vect{W}}^{-T} + \vect{W}^{-T}) \vect{W}^{-1}\|_F\\ \label{eq:stTrn1}
&\le \!\sigma ^2 d^2  \!\sqrt{\E  \|\vect{W} \!-\! \hat{\vect{W}}\|_F^2} \sqrt{\E \|\hat{\vect{W}}^{-1} \Sx (\hat{\vect{W}}^{-T} \!+\! \vect{W}^{-T}) \vect{W}^{-1}\|_F^2}\\
&\le \sigma ^2 d^2  \frac{\sqrt{2} d^{\frac{3}{2}} \|\Sx\|_F }{(a - (d-1)b)^3} \sqrt{\E  \|\vect{W} - \hat{\vect{W}}\|_F^2} \\ \label{eq:stTrn2}
&\le   \frac{\sqrt{2} d^{4.5}  \sigma ^2}{(a/d)d^2} \sqrt{\E  \|\vect{W} - \hat{\vect{W}}\|_F^2}
\end{align}
where~\eqref{eq:stTrn1} is due to the Cauchy–Schwarz inequality. For~\eqref{eq:stTrn2} note that $\|\Sx\|_F^2 \le d^2$ because all the entries of $\Sx$ are less than or equal to one.

We now proceed to upper bound $ \E \|\vect{W} - \hat{\vect{W}}\|_F^2$. Consider the following uniform quantizer: The diagonal entries are truncated at some $c>a$. The double segment $\pm [a, c]$ is divided into $l_1$ regions of width $\epsilon_1 = 2(c-a)/l_1$ each. For $|w|>c$ we take $\hat{w} = \text{sign}(w)c$. Therefore
	\begin{align}
	&\E (W_{11} - \hat{W}_{11})^2 \\
	&= \frac{Q(c)}{Q(a)}\E \left( \left. (W_{11} - \hat{W}_{11})^2 \right| |W_{11}|>c \right)  \\
	& \ + \left(1- \frac{Q(c)}{Q(a)} \right)\E \left( \left. (W_{11} - \hat{W}_{11})^2 \right| |W_{11}|<c \right) \\
	&\le \frac{Q(c)}{Q(a)}\E \left( \left. (W_{11} - c)^2 \right| |W_{11}|>c \right) + \epsilon_1^2\\
	&= \frac{Q(c)}{Q(a)}(1 + cs(c) + c^2)  + \epsilon_1^2\\  \label{eq:qntzTrn}
	&\le \frac{a + a^{-1}}{c} e^{-\frac{c^2 - a^2}{2}} 2(1 + c^2) + \epsilon_1^2\\
	&\le 8c^2 e^{-\frac{c^2 - a^2}{2}}  + \epsilon_1^2
	\end{align}
	where \eqref{eq:qntzTrn} is obtained with some manipulations on $t \le s(t) \le t + t^{-1}$. For the off-diagonal entries, the segment $[-b, b]$ is divided into $l_2$ regions of width $\epsilon_2 = 2b/l_2$ each. Therefore
	\begin{align}
	&\E (W_{12} - \hat{W}_{12})^2 \le \epsilon_2^2.
	\end{align}
It follows that
\begin{align}
&\E \|\vect{W} - \hat{\vect{W}}\|_F^2\\
&= d \E (W_{11} - \hat{W}_{11})^2 + (d^2-d) \E (W_{12} - \hat{W}_{12})^2\\
&\le 8 d c^2 e^{-\frac{c^2 - a^2}{2}} + d \epsilon_1^2 + d(d-1) \epsilon_2^2\\
&\le 8 d c^2 e^{-\frac{c^2 - a^2}{2}}  + d^2 (\epsilon_1 + \epsilon_2)^2.
\end{align}
We take $c = \sqrt{3}a$ and $l_1 = l_2$ and thus $\epsilon_1 + \epsilon_2 = 2(c-a+b)/l_1 \le 4a/l_1$. The number of bits used for quantization is $k_q = \log l_1$ and therefore $\epsilon_1 + \epsilon_2 \le  4a 2^{-k_q}$. Now,
\begin{align}
&\sqrt{\E \|\vect{W} - \hat{\vect{W}}\|_F^2}\\
&\le \sqrt{ 24 d a^2 e^{-a^2}  + 16 d^2 a^2 2^{-2 k_q}}\\
&\le  \sqrt{(5 a d)^2 (e^{-a^2}  +  2^{-2 k_q}) }\\ \label{eq:finalBnd}
&\le 5 a d ( e^{-\frac{a^2}{2}}  +  2^{-k_q}),
\end{align}
and finally, combining~\eqref{eq:finalBnd} with~\eqref{eq:ftTrn3} and~\eqref{eq:stTrn2} yields 
\begin{align}
&\E\|\hat{\Rho} - \Rho\|^2 -  \E\|\hat{\Rho}_0 - \Rho\|^2\\ 
&\le 25 d^6 (1 - \sigma^2)    ( e^{-\frac{a^2}{2}}  +  2^{-k_q})^2 +  5\sqrt{2} d^{4.5}  \sigma ^2 ( e^{-\frac{a^2}{2}}  +  2^{-k_q})\\
&\le 25 d^6 ( e^{-\frac{a^2}{2}}  +  2^{-k_q})
\end{align}
which completes the proof.

\subsection{Proof of Lemma \ref{lm:matrices}}\label{apn:mtrcsProof}
Denote by $\mathcal{P}$  the set of all $d \times d$ signed permutation matrices, i.e. matrices with exactly one nonzero entry in every row and every column, that takes values in $\{-1,1\}$. For any $d \times d$ matrix $B$ and any  $P \in \mathcal{P}$, the matrix $PBP^T$ is obtained from $B$ by performing the same permutation on the rows and columns of $B$, with possible sign changes. Specifically, the diagonal of $P B P^T$ is a permutation of the diagonal of $B$, and the off-diagonal of  $P B P^T$ is a permutation of the off-diagonal entries of $B$ with possible sign changes. 

Suppose that the random matrix $\vect{N}$ has the same distribution as $P \vect{N} P^T$ for any $P \in \mathcal{P}$. It follows that  $\E \vect{N}$ must be a scalar multiple of the identity matrix since for any  $i \neq j$ there exist two matrices $P_1, P_2 \in \mathcal{P}$ such that 
\begin{enumerate}
	\item for some $i' \neq j'$,
	\begin{align}
	&(P_1 \vect{N} P_1^T)_{i'j'} = (\vect{N})_{ij}\\  
	&(P_2 \vect{N} P_2^T)_{i'j'} = -(\vect{N})_{ij}
	\end{align}
	and thus $(\E \vect{N})_{ij} = -(\E \vect{N})_{ij}$.
	\item for some $i'$,
	\begin{align}
	&(P_1 \vect{N} P_1^T)_{i'i'} = (\vect{N})_{ii}\\  
	&(P_2 \vect{N} P_2^T)_{i'i'} = (\vect{N})_{jj}
	\end{align}
	and thus $(\E \vect{N})_{ii} = (\E \vect{N})_{jj}$.
\end{enumerate}

The assumptions in the lemma imply that $\vect{M}$ and $P \vect{M} P^T$ have the same distribution for any $P \in \mathcal{P}$, thus $\vect{M} \vect{M}^T$ and $(P \vect{M} P^T)(P \vect{M} P^T)^T$ have the same distribution. Hence 
\begin{align}
(P \vect{M} P^T) (P \vect{M} P^T)^T = P \vect{M} \vect{M}^T P^T
\end{align}
 and thus $P \vect{M} \vect{M}^T P^T$ has the same distribution as $\vect{M} \vect{M}^T$. This implies that $(P \vect{M} \vect{M}^T P^T)^{-1}$ has the same distribution as $(\vect{M} \vect{M}^T)^{-1}$, and since 
\begin{align}
(P \vect{M} \vect{M}^T P^T)^{-1} = P (\vect{M} \vect{M}^T )^{-1} P^T
\end{align}
 we have that $P (\vect{M} \vect{M}^T )^{-1} P^T$ has the same distribution as $(\vect{M} \vect{M}^T )^{-1}$. Therefore  $\E (\vect{M} \vect{M}^T )^{-1}$ is a scalar multiple of the identity matrix.

\subsection{Proof of Proposition~\ref{pr:paretoQntz}}\label{apn:proofParetoProp}
For any distribution on $X$, and for any $u, \Delta$,
\begin{align}
\E &(\hat{\rho}_{\thq} - \rho)^2 
\!=\! \rho^2 \E \left( \frac{X_J - \hat{X}_J }{\hat{X}_J} \right)^2 \!+\! (1 - \rho^2) \E \frac{1}{\hat{X}_J^2} \\
&\le  \frac{\rho^2}{t^2} \E \left( X_J - \hat{X}_J  \right)^2 +   \frac{1 - \rho^2}{t^2} \\
&=   \frac{\rho^2}{t^2} \Pr (X_J < u) \E  \left(  \left. \left( X_J - \hat{X}_J  \right)^2 \right| X_J < u \right)\\ 
\nonumber & \quad + \frac{\rho^2}{t^2} \Pr (X_J > u) \E  \left( \left. \left( X_J - u  \right)^2 \right| X_J > u \right)
+ \frac{1 - \rho^2}{t^2} \\
&\le   \frac{\rho^2}{t^2}  \Delta^2
+ \frac{\rho^2}{t^2} \frac{\Pr (X > u)}{\Pr (X > t)} \E \left( \left. \left( X - u  \right)^2 \right| X > u \right)
+ \frac{1 - \rho^2}{t^2}.
\end{align}
In this Pareto example we have 
\begin{align}
\E \left( \left. \left( X - u  \right)^2 \right| X > u \right) = c u^2 
\end{align}
where $c = 2/((\alpha - 1)^2 (\alpha - 2))$, and thus
\begin{align}\label{eq:thrshQntzdErr}
\E (\hat{\rho}_{\thq} - \rho)^2 
&\le  \frac{1 - \rho^2 + \rho^2 \Delta^2}{t^2} +  c \rho^2 \left( \frac{t}{u} \right)^{\alpha-2} .
\end{align}
We take $u = t^{\frac{\alpha}{\alpha - 2}}$, and thus \eqref{eq:thrshQntzdErr} becomes
\begin{align}\label{eq:thrshQntzdErr1}
\E (\hat{\rho}_{\thq} - \rho)^2 
&\le  \frac{1 - \rho^2 + \rho^2 \Delta^2 + c \rho^2}{t^2}.
\end{align}
The bits are allocated by
\begin{align}
&k_q = \frac{1}{\alpha-1}k, \quad k_l = \frac{\alpha-2}{\alpha-1}k,
\end{align}
and the threshold $t$ is determined by $k_l$ from the solution of $k_l = h_g(\Pr(X>t))$, which yields $t = 2^{\frac{k_l}{\alpha} (1 - o(1))}$. We have $\Delta \le 1$ since
\begin{align}
\nonumber \Delta &= 2^{-k_q}(u-t)= 2^{-k_q}(t^{\frac{\alpha}{\alpha - 2}}-t)= 2^{-k_q} t^{\frac{\alpha}{\alpha - 2}} (1 - t^{\frac{-2}{\alpha - 2}})\\
\nonumber & = 2^{-\frac{k}{\alpha-1}}  2^{\frac{k}{\alpha-1}(1 - o(1))}   (1 - t^{\frac{-2}{\alpha - 2}}) = 2^{-o(k)}  (1 - t^{\frac{-2}{\alpha - 2}}).
\end{align}
Note that for $\alpha > 3$ we have $c<1$ and thus
\begin{align}
\E (\hat{\rho}_{\thq} - \rho)^2 
&\le  \frac{1 + \rho^2}{t^2} \le  \frac{1 + \rho^2}{2^{\frac{2 }{\alpha}  \frac{\alpha-2}{\alpha-1}k(1 - o(1))}}.
\end{align}

\subsection{Proof of Corollary~\ref{cor:cor-hid2}}\label{apn:cor-hid2-proof}
	Set any real-valued sequence $\{\alpha_k\}_{k=1}^\infty$ such that all the elements are distinct, and $\sum\alpha_k^2 =1$. Pick $Z = \sum_k \alpha_k Z_k$, where $Z_k\sim p_X$ are i.i.d and mutually independent of $X$. Then $Y = \alpha X  + \sum_k \alpha_k Z_k$ is a weighted sum of i.i.d. r.v.s, hence knowing $p_X$ and $p_Y$, or even knowing all the weights $\alpha, \{\alpha_k\}$, there is no way to distinguish between the case where $X$ has coefficient $\alpha$ and where $X$ has coefficient $\alpha_k$ for some $k$. Thus, there is an infinite number of possible correlations. 

\subsection{Maximum likelihood approximation}\label{apn:MLE}
In this section we provide further justification for the estimator $\hat{\Rho}= \vect{Y}_J/\E X_J$ (or $Y_J/\E X_J$ in the scalar setup which is a special case), by showing that $\vect{Y}_J/ X_J$ is an approximation of the maximum likelihood estimator. The model is 
\begin{align}
\vect{Y}_J = \Rho X_J + \Sigma^{\frac{1}{2}} \vect{Z}_J
\end{align}
where either $J = \argmax_i \{X_i\}$ if we use the max method, or, if we use the threshold method, $J = \min\{i:X_i>t\}$. We wish to maximize $f_{X_J \vect{Y}_J}$ which is equivalent to maximizing $f_{\vect{Y}_J|X_J}$, since $f_{X_J}$ does not depend on $\Rho$. It follows that the actual distribution of $X_J$ is irrelevant. We have $\vect{Y}_J|X_J \sim \mathcal{N}(\Rho X_J, \Sigma)$
with $\Sigma = \Sigma_{\vect{Y}} - \Rho \Rho^T$ and thus 
\begin{align}
&\frac{\partial}{\partial \Rho} \ln f_{X_J \vect{Y}_J} \\
\nonumber &= -\frac{1}{2}\frac{\partial}{\partial \Rho} \left( \ln \det \Sigma  + (\vect{Y}_J - \Rho X_J)^T \Sigma^{-1} (\vect{Y}_J - \Rho X_J) \right)\\
\nonumber &=\Sigma^{-1} \left( \Rho  -  (\vect{Y}_J - \Rho X_J) \left( (\vect{Y}_J - \Rho X_J)^T \Sigma^{-1} \Rho - X_J  \right)  \right)
\end{align}
meaning we want to solve
\begin{align}
\Rho =  (\vect{Y}_J - \Rho X_J) \left( (\vect{Y}_J - \Rho X_J)^T \Sigma^{-1} \Rho - X_J  \right).
\end{align}
Note that the rightmost term is a scalar and thus the solution must be of the form $\hat{\Rho}_{\textnormal{ML}} = C \vect{Y}_J$ where $C$ is a scalar that depends on $X_J$ and $\vect{Y}_J$. Plugging it yields that $C$ is obtained as the solution of the third degree polynomial
\begin{align}
\nonumber \vect{Y}_J^T\Sigma_{\vect{Y}}^{-1} \vect{Y}_J C^2 (X_J-C) \!-\! (X_J^2 \!-\! 1 \!+\! \vect{Y}_J^T\Sigma_{\vect{Y}}^{-1} \vect{Y}_J)C \!+\! X_J \!=\! 0.
\end{align}
In our setups $X_J$ takes large values. This implies in general that the entries of $\vect{Y}_J$ are also large, and thus $C$ should be small as we expect $\hat{\Rho}_{\textnormal{ML}} = C \vect{Y}_J$ to produce moderate values. Therefore we can assume that $X_J-C \approx X_J$ and that $X_J^2 - 1 \approx X_J^2$, which results in a quadratic equation in $C$ whose solutions are $X_J/(\vect{Y}_J^T\Sigma_{\vect{Y}}^{-1} \vect{Y}_J)$ and $1/X_J$. Note that (with either max or threshold) $\Var X_J$ approaches zero as the number of bits increases, and therefore the loss in replacing $X_J$ with $\E X_J$ is negligible (it is also evident in the optimality claims throughout where it is shown that the estimators, which do not use the actual value of $X_J$, achieve the CRLB that assumes $X_J$ is known).

\nocite{ourISITversion}

\bibliography{bibtex_refs}
\bibliographystyle{ieeetr}

%\end{spacing}

\end{document}